\documentclass[11pt, reqno, english]{amsart} 
\usepackage[utf8]{inputenc}
\usepackage[T1]{fontenc}
\usepackage{amsmath,amsthm}
\usepackage{amsfonts,amssymb}
\usepackage{url}
\usepackage{mathtools}  
\usepackage[colorlinks=true,urlcolor=blue,linkcolor=red,citecolor=magenta]{hyperref}
\usepackage{enumerate, paralist}

\usepackage{charter}

\usepackage[left=1in,right=1in,top=1in,bottom=1in]{geometry}
\numberwithin{equation}{section}

\theoremstyle{plain}
\newtheorem{theorem}{Theorem}[section]
\newtheorem{lemma}[theorem]{Lemma}
\newtheorem{corollary}[theorem]{Corollary}
\newtheorem{proposition}[theorem]{Proposition}

\theoremstyle{definition}

\newtheorem{remark}[theorem]{Remark}

\newcommand{\R}{\mathbb{R}}
\newcommand{\C}{\mathbb{C}}
\newcommand{\Z}{\mathbb{Z}}
\newcommand{\K}{\mathbb{K}}

\newcommand{\supp}{\mathrm{supp}}

\DeclareMathOperator{\conv}{\mathrm{conv}}

\DeclareMathOperator{\sgn}{\mathrm{sgn}}
\DeclareMathOperator{\avg}{\mathrm{avg}}
\DeclareMathOperator{\dist}{\mathrm{dist}}
\DeclareMathOperator{\tens}{\mathrm{tens}}

\title{Topological Tverberg Theorems for Products of Polytopes}

\author[Simon]{Steven Simon}
\address[SS]{Dept.\ Math.\, Bard College, Annandale-on-Hudson, NY 12504, USA}
\email{ssimon@bard.edu} 

\begin{document} 

\begin{abstract}

The topological Tverberg theorem asserts that if $r$ is a prime power then for any continuous map $f\colon \Delta_{(r-1)(d+1)}\rightarrow \R^d$ from the $(r-1)(d+1)$-dimensional simplex $\Delta_{(r-1)(d+1)}$ to $\R^d$ there exist $r$ pairwise disjoint faces of the simplex whose images have non-empty $r$-fold intersection. By refinement, the same conclusion holds if the simplex is replaced by any polytope of the same dimension. While this dimension is tight for simplices, recent work of Sober\'on and Zerbib shows that this need not be true for polytopes in general. 

Here we give topological Tverberg theorems for products of simplices. Each of these improves upon the $(r-1)(d+1)$-dimensional threshold, even while imposing the structural condition that the ``Tverberg faces'' of the product are themselves the products of pairwise disjoint faces from each simplex factor. As before, refinement extends these results, and in particular their dimensional improvements, to products of arbitrary polytopes. As an example, if $d+1$ is a power of two then whenever $m\geq n\geq d+1$ and $m+n=3d+2$ we show that any continuous map $f\colon \Delta_m\times \Delta_n\rightarrow \R^d$ admits disjoint faces $\sigma_1,\sigma_2$ of $\Delta_m$ and $\tau_1,\tau_2$ of $\Delta_n$ such that $\cap_{i,j\in[2]} f(\sigma_i\times\tau_j)\neq \emptyset$. In the case of multilinear maps, our results imply partitions of grid-indexed point sets in $\R^d$ by specialized subsets with stronger intersection conclusions than given by Tverberg's original theorem. Lastly, we extend our results to van Kampen--Flores type theorems which impose dimensional restrictions on the faces of each product factor. 
\end{abstract}

\maketitle

\section{Introduction and Statement of Results}
\label{sec:intro}

\subsection{Topological Tverberg Theorems for Polytopes}
A celebrated theorem of Tverberg~\cite{Tv66} from 1966 asserts that any set of at least $(r-1)(d+1)+1$ points in $\R^d$ admits a partition by $r$ pairwise disjoint subsets whose convex hulls completely overlap. The result is a cornerstone of discrete geometry, and has given rise to a wealth of profound extensions and variants over the course of the past sixty years (see, e.g., the recent survey~\cite{BaSo19} or ~\cite{Ma08}). Particularly well-known among these is a topological generalization, which begins by reformulating Tverberg's theorem as the statement that for any linear map from the $(r-1)(d+1)$-dimensional simplex $\Delta_{(r-1)(d+1)}$ to $\R^d$ there exist $r$ pairwise disjoint faces of the simplex whose images completely overlap. The topological Tverberg of Oz\"aydin~\cite{Oz87} and Volovikov~\cite{Vo96} shows that one may replace linear maps with arbitrary continuous ones whenever $r$ is a prime power, a condition on $r$ which was first shown to be necessary by Frick~\cite{Fr15}.

\begin{theorem}
\label{thm:top Tverberg} 
Let $d\geq 1$ be an integer and let $r\geq 2$ be a prime power. Then for any continuous map $f\colon \Delta_{(r-1)(d+1)}\rightarrow \R^d$ there exists $r$ pairwise disjoint faces $\sigma_1,\ldots, \sigma_r$ of $\Delta_{(r-1)(d+1)}$ such that $\cap_{i=1}^r f(\sigma_i)\neq \emptyset$.
\end{theorem}

As observed in~\cite{HKTT23}, a result of Gr\"unbaum~\cite{Gr03} that any (convex) polytope is a refinement of a simplex of the same dimension easily implies an extension of the topological Tverberg theorem to arbitrary polytopes $P$ of dimension $(r-1)(d+1)$. Namely, one has a homeomorphism $h\colon P\rightarrow \Delta_{(r-1)(d+1)}$ such that for each face $\sigma$ of $\Delta_{(r-1)(d+1)}$ the pullback $h^{-1}(\sigma)$ is the union of faces of $P$. Composing $h^{-1}$ with any continuous map $f\colon P\rightarrow \R^d$ and applying Theorem~\ref{thm:top Tverberg} thereby gives $r$ pairwise disjoint faces of $P$ whose images under $f$ completely overlap. 

\begin{corollary} 
\label{cor:polytope Tverberg}
Let $d\geq 1$ be an integer, let $r\geq 2$ be a prime power, and let $P$ be a polytope of dimension $(r-1)(d+1)$. Then for any continuous map $f\colon P\rightarrow \R^d$ there exists $r$ pairwise disjoint faces $\sigma_1,\ldots, \sigma_r$ of $P$ such that $\cap_{i=1}^r f(\sigma_i)\neq \emptyset$.
\end{corollary}

 We shall call $\sigma_i$ such as those from Corollary~\ref{cor:polytope Tverberg} \emph{Tverberg faces} and the set $\{\sigma_1,\ldots, \sigma_r\}$ of such faces an \emph{$r$-Tverberg family}.

The dimension $t(r,d):=(r-1)(d+1)$ of the simplex of the topological Theorem cannot be lowered for given $r$ and $d$, even for generic linear maps (see, e.g.,~\cite{PS14}). Nonetheless, one may ask for dimensional improvements of Corollary~\ref{cor:polytope Tverberg} in general -- that is, for given $r$ and $d$, to seek polytopes $P$ with $\dim P<t(r,d)$ such that any continuous map $f\colon P\rightarrow \R^d$ still admits an $r$-Tverberg family. Results along these lines were recently given in ~\cite{SoZe26} when $r$ is prime for certain classes of polytopes, including cross-polytopes and polytopes with small face diameter, as well as in the linear setting for arbitrary integers $r\geq 3$ in the case of cyclic polytopes and neighborly polytopes more generally.

\subsection{Topological Tverberg Theorems for Products} This paper presents topological Tverberg theorems for products of polytopes. Each of these improves upon the $t(r,d)$-dimensional Tverberg threshold, even while guaranteeing the structural condition that the Tverberg faces of the product are the product of pairwise disjoint faces from each polytope factor. When restricted to multilinear maps defined on products of simplices, our results yield Tverberg partitions of grid-indexed point sets in $\R^d$ by subsets with greater structure and stronger intersection conclusions than those guaranteed Tverberg's original theorem (see Corollary~\ref{cor:multilinear}). In Remark~\ref{rem:alternative} we give a comparison of our results to an alternative recent product Tverberg theorem given in ~\cite{HMMZ26}.  

Let $p\geq 2$ be a prime number and let $i_1\leq \cdots \leq i_k$ be a non-trivial partition of an integer $\ell\geq 2$. Given polytopes $P_1,\ldots, P_k$, let $P=P_1\times\cdots\times P_k$ be their product and let $f\colon P\rightarrow \R^d$ be a continuous function. Supposing that for each $j\in [k]$ one has a family $\{\sigma_{m_j}^j\mid m_j\in [p^{i_j}]\}$ of $p^{i_j}$ pairwise disjoint faces of $P_j$, the resulting family \[\{\sigma_m:=\sigma_{m_1}^1\times\cdots \times \sigma_{m_k}^k\mid m=(m_1,\ldots, m_k)\in [p^{i_1}]\times\cdots\times [p^{i_k}]\}\] consists in particular of $p^\ell$ pairwise disjoint faces of $P$, each of which lies in the boundary product $\partial P_1\times\cdots \times \partial P_k$. The $\sigma_m$ will be said to form a \emph{$p^\ell$-Tverberg family of type $(p^{i_1},\ldots, p^{i_k})$} provided the $p^\ell$-fold intersection of the images of the $\sigma_m$ is non-empty. 

We shall say that a $k$-tuple $(n_1,\ldots, n_k)$ is \emph{$(p^\ell,d;p^{i_1},\ldots, p^{i_k})$-Tverberg admissible} provided that any continuous map $f\colon \Delta_{n_1}\times \cdots \times \Delta_{n_k}\rightarrow \R^d$ admits a $p^\ell$-Tverberg family of type $(p^{i_1},\ldots, p^{i_k})$. As with Corollary~\ref{cor:polytope Tverberg}, taking products of refinements then implies that one has  $p^\ell$-Tverberg families of type $(p^{i_1},\ldots, p^{i_k})$ for any continuous map $f\colon P_1\times\cdots\times P_k \rightarrow \R^d$, where each $P_j$ is any polytope of dimension $n_j$.

The sharpness of the topological Theorem~\ref{thm:top Tverberg} immediately implies that for each $j\in [k]$ one must have $n_j\geq (p^{i_j}-1)(d+1)$ for any $(p^\ell,d;p^{i_1},\ldots, p^{i_k})$-Tverberg admissible $k$-tuple $(n_1,\ldots, n_k)$.  As discussed in Remark~\ref{rem:dimension} below, elementary obstruction theory imposes dimensional requirements on the equivariant framework underlying the existence of such Tverberg families. These lead to an expectation that the $n_j$ must also satisfy the sum condition  \[\sum_{j=1}^k n_j\geq t(p^\ell, d; p^{i_1},\ldots, p^{i_k}):=d(p^\ell-1)+\sum_{j=1}^k (p^{i_j}-1). \] It is easily verified that $t(p^\ell,d;p^{i_1},\ldots, p^{i_k})<t(p^\ell,d)$ in all cases. 

In what follows, we give Tverberg admissible $k$-tuples $(n_1,\ldots, n_k)$ satisfying the sum condition above, and so in particular one has dimensionally improved topological Tverberg theorems for products $\Delta_{n_1}\times\cdots\times\Delta_{n_k}$ of $n_j$-dimensional simplices. Owing to the obstruction considerations mentioned before, we expect that these results are dimensionally tight. As with Theorem~\ref{thm:top Tverberg}, by refinement each of these results extends to a corresponding dimensionally improved topological Tverberg theorem for products $P=P_1\times\cdots\times P_k$ of arbitrary $n_j$-dimensional polytopes $P_j$. 

\subsection{Statement of Results} 

 Our main theorem guarantees admissible Tverberg $k$-tuples for arbitrary dimensions $d$ and all non-trivial partitions $1\leq i_1\leq \cdots \leq i_k$ of any integer $\ell\geq 2$.

\begin{theorem}
\label{thm:main}
Let $d\geq 1$ be an integer, let $p\geq 2$ be a prime number, and let $1\leq i_1\leq \cdots \leq i_k$ be a non-trivial partition of an integer $\ell\geq 2$. Suppose that $C_1,\ldots, C_k$ is a partition of $[\ell]$ where each $C_j$ has size $i_j$, and for each $j\in[k]$ let $n_j=d(p-1)\cdot\sum_{a\in C_j} p^{a-1}+p^{i_j}-1$.  Then any continuous function $f\colon \Delta_{n_1}\times\cdots\times \Delta_{n_k}\rightarrow \R^d$ admits a $p^\ell$-Tverberg family of type $(p^{i_1},\ldots, p^{i_k})$. 
\end{theorem}

As $(p-1)\cdot \sum_{j=0}^{\ell-1} p^j=p^\ell-1$, one has $\sum_{j=1}^k n_j=t(p^\ell,d;p^{i_1},\ldots, p^{i_k})$ in Theorem~\ref{thm:main}. Owing to the uniqueness of $p$-ary expansions, the $k$-tuples corresponding to the various $C_1,\ldots, C_k$ are all distinct, so one has at least $\binom{\ell}{i_1,\ldots, i_k}$ Tverberg admissible $k$-tuples for each non-trivial partition of $\ell$. 

As an example of Theorem~\ref{thm:main}, consider the partition $i_1=1$ and $i_2=2$ of $\ell=3$. Letting $C_1=\{i\}$ and $C_2=[3]\setminus\{i\}$ for all $i\in [3]$, Theorem~\ref{thm:main} guarantees an $8$-Tverberg family of type $(2,4)$ for any continuous map $f\colon  \Delta_{n_1}\times\Delta_{n_2}\rightarrow \R^3$ with $(n_1,n_2)\in \{(4,21), (7,18), (13,12)\}$. Here $t(8,3;2,4)=25$ can be compared to $t(8,3)=28$.\\ 

 Owing to the equivariant cohomological machinery underlying our methods, the Tverberg admissible $k$-tuples we give outside of those guaranteed by Theorem~\ref{thm:main} depend on number-theoretic conditions on the dimension $d$. This is a general phenomenon for results in topological combinatorics, for instance as seen in measure equipartition theory. In fact, when $p=2$ our methods are closely related to those~\cite{BFHZ16,BFHZ18, MLVZ06} used in perhaps the most central problem in this area, the Gr\"unbaum--Hadwiger--Ramos equipartition problem~\cite{Gr60, Ha66, Ra96}, which asks for the minimum dimension $m:=\Delta(d,\ell)$ such that any $d$ masses in $\R^m$ can be equipartitioned by $\ell$ affinely independent hyperplanes. Our results for odd primes are similarly associated with a complex variant of that problem for regular fans~\cite{Si15}. 
 
When $p=2$ we give a recursive formula for admissible Tverberg $k$-tuples for arbitrary integer partitions $1\leq i_1\leq\cdots\leq i_k$ of $\ell$ depending on the decomposition of $d$ in the form $d=2^{s_1}+t_1$, where $s_1$ and $t_1$ are integers satisfying $s_1\geq 0$ and $0\leq t_1<2^{s_1}$. We recursively obtain numbers $s_2,\ldots, s_{\ell-1}$, $t_2,\ldots, t_{\ell-1}$ as follows. For $2\leq j \leq \ell-1$, given $s_{j-1}\geq 0$ and $0\leq t_{j-1}<2^{s_{j-1}}$ we define $s_j\geq0$ and $0\leq t_j<2^{s_j}$ by the equation $2^{s_{j-1}}+2t_{j-1}=2^{s_j}+t_j$. We then define $u_1,\ldots, u_\ell$ by \[u_1=2^{s_1+\ell-1}+t_1, u_2=2^{s_2+\ell-2}+t_2,\ldots, u_{\ell-1}=2^{s_{\ell-1}+1}+t_{\ell-1},\,\, \text{and} \,\, u_\ell=2^{s_{\ell-1}}+2t_{\ell-1}.\] It should be observed that $\Delta(2^{s_1}+t_1,\ell)\leq u_1$ is the current best general upper bound~\cite{MLVZ06, BCCD23} to the Gr\"unbaum--Hadwiger--Ramos hyperplane equipartition problem above.

\begin{theorem} 
\label{thm:p=2 general} 
Let $1\leq i_1\leq \cdots \leq i_k$ be a non-trivial partition of an integer $\ell\geq 2$ and let $C_1,\ldots, C_k$ be a partition of $[\ell]$ by subsets of size $i_j$ each. Let $s_1\geq 0$ and $0\leq t_1<2^{s_1}$ be integers and let $d=2^{s_1}+t_1$. For $u_1,\ldots, u_\ell$ as above, define $n_j=\sum_{a\in C_j} u_a+2^{i_j}-1$ for all $j\in [k]$. Then any continuous map $f\colon \Delta_{n_1} \times\cdots\times \Delta_{n_k} \rightarrow \R^d$ admits a $2^\ell$-Tverberg family of type $(2^{i_1},\ldots, 2^{i_k})$. 
\end{theorem}

Our proof of Theorem~\ref{thm:p=2 general} will show that $\sum_{j=1}^k n_j=t(2^\ell,d;2^{i_1},\ldots, 2^{i_k})$. Considering the partition $i_1=1$ and $i_2=2$ of $\ell=3$ as before, an application of Theorem~\ref{thm:p=2 general} to the partition $C_1=\{i\}$ and $C_2=[3]\setminus\{i\}$ of $[3]$ for any $i\in [3]$ shows that one has an $8$-Tverberg family of type $(2,4)$ for any continuous map $f\colon \Delta_{n_1}\times \Delta_{n_2}\rightarrow \R^3$ when $(n_1,n_2)\in \{(5,20), (9,16), (10,15)\}$. Each of these pairs is distinct from those given by Theorem~\ref{thm:main}.\\ 

Our final admissibility results hold in the symmetric setting of $p^{2\ell}$-Tverberg families of type $(p^\ell,p^\ell)$. We shall say that admissible pairs $(n_1,n_2)$ here are \emph{balanced} provided  $n_1+n_2=t(p^{2\ell},d;p^\ell,p^\ell)$ and $|n_1-n_2|\leq 1$, or in other words that the dimensions of the simplex factors are as close together as possible given the sum condition. At the other extreme, we say that $(p^{2\ell},d;p^\ell,p^\ell)$-Tverberg pairs are \emph{arbitrarily prescribable} if $(n_1,n_2)$ is admissible whenever $n_1,n_2\geq (p^{\ell}-1)(d+1)$ and $n_1+n_2=t(p^{2\ell},d;p^\ell,p^\ell)$. 

Theorem~\ref{thm:binomial} below gives a simple binomial coefficient test for $(p^2,d;p,p)$-admissibility, from which it follows that $(4,d;2,2)$-Tverberg pairs are arbitrarily prescribable when $d+1$ is a power of two, and so in particular that the pair $(\lceil\frac{3d}{2}\rceil+1,\lceil\frac{3d}{2}\rceil)$ is balanced. For odd primes $p$ this test gives an evenness criterion on the $p$-ary expansion of $d(p-1)/2$ for the existence of balanced $(p^2,d;p,p)$-Tverberg pairs. Additional balanced $(4,d;2,2)$-Tverberg pairs exist when $d$ is a power of two, and one has balanced $(p^{2\ell},d;p^\ell,p^\ell)$-Tverberg pairs as well when (1) $p=3$ and $d$ is twice a power of three and when (2) $p=5$ and $d$ is a power of $5$. 

\begin{theorem}
\label{thm: arbitrary/balanced} Let $d\geq 1$ be an integer.
\begin{compactenum}[(a)]

\item If $d+1$ is a power of two, then any continuous map $f\colon \Delta_{n_1}\times\Delta_{n_2}\rightarrow \R^d$ admits a $4$-Tverberg family of type $(2,2)$ whenever $n_1,n_2\geq d+1$ and $n_1+n_2=3d+2$.

\item Suppose that $d\geq 2$ is a power of two and let $n=\frac{3d}{2}+1=t(4,d;2,2)/2$. Then any continuous map $f\colon \Delta_n\times\Delta_n\rightarrow \R^d$ admits a $4$-Tverberg family of type $(2,2)$. 

\item Let $p$ be an odd prime. Suppose that the $p$-ary coefficients of $d(p-1)/2$ are all even and let $n=(p-1)(\frac{d(p+1)}{2}+1)=t(p^2,d;p,p)/2$. Then any continuous map $f\colon \Delta_n\times\Delta_n\rightarrow \R^d$ admits a $p^2$-Tverberg family of type $(p,p)$.

\item Let $\ell\geq 1$ be an integer, let $p\in \{3,5\}$, and suppose that $d(p-1)/2=2p^s$ for some $s\geq 0$. Let $n=(p^\ell-1)(\frac{d(p^\ell+1)}{2}+1)=t(p^{2\ell},d;p^\ell,p^\ell)/2$. Then any continuous map $f\colon \Delta_n\times\Delta_n\rightarrow \R^d$ admits a $p^{2\ell}$-Tverberg family of type $(p^\ell,p^\ell)$. 
\end{compactenum}
\end{theorem}

Theorem~\ref{thm: arbitrary/balanced}(b) combined with Theorem~\ref{thm:main} implies in particular that $(4,d;2,2)$-Tverberg pairs are arbitrarily prescribable when $d=2$. As examples of Theorem~\ref{thm: arbitrary/balanced}(c)(d), we have that any continuous map $f\colon  \Delta_{10}\times  \Delta_{10}\rightarrow \R^2$ admits a $9$-Tverberg family of type $(3,3)$, while any continuous map $f\colon \Delta_{88}\times  \Delta_{88}\rightarrow \R^2$ admits an $81$-Tverberg family of type $(9,9)$. Here $t(9,2;3,3)=20$ and $t(81,2;9,9)=176$, which can be compared to $t(9,2)=24$ and $t(81,2)=240$, respectively. 

 We note that Theorem~\ref{thm: arbitrary/balanced}(c) holds for all $d$ of the form $d=2a(p^{\ell_1}+\cdots+p^{\ell_k})$, where the $\ell_1,\ldots, \ell_k$ are all distinct and $1\leq a\leq p-1$ is odd. In particular, one may let $d=2a(p^\ell-1)/(p-1)$ for any $\ell\geq 1$. Moreover, if $p\equiv 1~(\text{mod}~4)$ one may also let $d$ be the sum of distinct powers of $p$, and so in particular let $d=(p^\ell-1)/(p-1)$ for any $\ell\geq 1$ (see ~\cite{HMRS26}).

Finally, we observe that if $d-1$ is a power of two and $n=\lceil\frac{3d}{2}\rceil+1$, then one also has that any continuous map $f\colon \Delta_n\times\Delta_n\rightarrow \R^d$ admits a $4$-Tverberg family of type $(2,2)$ (see Remark~\ref{rem:d=2^s+1} below). As $2n=3d+3=t(4,d)>t(4,d;2,2)$, the pair $(n,n)$ is not optimal, however. We note that the numbers $n=\lceil\frac{3d}{2}\rceil+1$ in our results when either $d-1,d,$ or $d+1$ is a power of two correspond to the only known solutions to the Gr\"unbaum--Hadwiger--Ramos problem above for two hyperplanes, in which case one has $\Delta(d,2)=\lceil\frac{3d}{2}\rceil$.

\subsection{The multilinear setting} When restricted to maps $f\colon \Delta_{n_1}\times\cdots \times \Delta_{n_k}\rightarrow \R^d$ which are linear on each simplex factor, the results above give specialized Tverberg partitions of grid-indexed point sets in $\R^d$. Namely, suppose that $(n_1,\ldots, n_k)$ is  $(p^\ell,d;p^{i_1},\ldots, p^{i_k})$-Tverberg admissible and let $X=\{x_a\mid a=(a_1,\ldots, a_k)\in \prod_{j=1}^k [n_j+1]\}$ be any set of $\prod_{j=1}^k (n_j+1)$ points in $\R^d$. Given non-empty subsets $A^1\subset [n_1+1],\ldots, A^k\subset [n_k+1]$, let $A=A^1\times\cdots\times A^k$ and let $X_A=\{x_a\mid a\in A\}$ denote the corresponding subset of $X$. We define its \emph{tensor hull} by \[\tens(X_A)=\left\{\sum_{(a_1,\ldots, a_k)\in A}t_{a_1}\cdots t_{a_k} x_{(a_1,\ldots, a_k)}\mid t_{a_j}\geq 0\,\, \text{and}\,\, \sum_{a_j\in A^j}t_{a_j}=1\,\, \text{for all}\,\, j\in [k]\right \}.\] We observe that $\tens(X_A)$ is always a subset of $\conv (X_A)$ and is generally smaller. One then has the following immediate consequence of Theorems~\ref{thm:main}-~\ref{thm: arbitrary/balanced}.

\begin{corollary}
\label{cor:multilinear}
 Let $d\geq 1$ be an integer, let $p$ be a prime number, let $1\leq i_1\leq \cdots\leq i_k$ be a non-trivial partition of an integer $\ell\geq 2$, and let $(n_1,\ldots, n_k)$ be as in Theorems~\ref{thm:main}-~\ref{thm: arbitrary/balanced}. Then for each $j\in [k]$ there exists a partition of $[n_j+1]$ by $p^{i_j}$ pairwise disjoint subsets $A^j_1,\ldots, A^j_{p^{i_j}}$ such that \[\bigcap_{(m_1,\ldots, m_k)\in [p^{i_1}]\times\cdots\times [p^{i_k}]}\tens(X_{A^1_{m_1}\times\cdots \times A^k_{m_k}})\neq \emptyset.\] 
 \end{corollary}
 
 For example, Theorem~\ref{thm: arbitrary/balanced}(b) implies that any set $X$ of $25$ points in the plane indexed by the grid $[5]\times [5]$ admits partitions $\{A_1,A_2\}$ and $\{B_1,B_2\}$ of $[5]$ such that the corresponding four sets $X_{A_i\times B_j}$ partition $X$ and have completely overlapping tensor hulls. Theorem~\ref{thm:main} gives the analogous conclusion for any set of $24$ points in the plane indexed by $[4]\times [6]$. 
 
\begin{remark}
\label{rem:alternative} 
It should be mentioned that an alternative Tverberg product theorem was recently given in~\cite{HMMZ26} where the domain is a simplex with vertex set the grid $[n]^m$. For any dimension $d$ and prime power $r$ it was shown that when $n=\lceil(\frac{d}{m}+1)(r-1)+\frac{1}{m}\rceil$ there exists a partition $A_1,\ldots, A_r$ of $[n]$ and an integer $i\in [m]$ such that the images of the $r$ faces with vertex sets of the form $B_k:=[n]\times\cdots \times A_k\times \cdots\times [n]$ have overlapping images, where each $A_k$ lies in the $i$-th factor. The linear version of this theorem gives partitions of point sets indexed by such grids into subsets determined by the $B_k$ whose ordinary convex hulls all intersect. By contrast, the domains of our results are products of simplices, and our product Tverberg faces are determined by a partition of the vertex set of each simplex factor. For grid-indexed point sets we have a partition into subsets determined by a partition of each factor, with a resulting intersection of restricted convex hulls. The number of points we require for this additional structure is far larger than those given there (or in Tverberg's original theorem, which is the $m=1$ case), however. 
\end{remark}

\subsection{van Kampen--Flores type variants} As with the many constrained extensions of the topological Tverberg theorem (see, e.g., ~\cite{BFZ14}), one may provide various restrictions on the faces of our Tverberg families. We shall content ourselves with an analogue of a dimensionally constrained version of Sarkaria~\cite{Sa91} generalizing the well-known van Kampen--Flores theorem~\cite{vK32,Fl32}.

\begin{theorem}
\label{thm:vK-F}
Let $d\geq 1$ be an integer and let $r\geq 2$ be a prime power. Then any continuous map $f\colon \Delta_{(d+2)(r-1)}\rightarrow \R^d$ admits $r$ Tverberg faces $\sigma_1,\ldots, \sigma_r$ such that $\dim \sigma_i\leq \lfloor \frac{(d+1)(r-1)}{r}\rfloor$ for all $i\in [r]$. 
\end{theorem}

As refinements preserve dimension, Theorem~\ref{thm:vK-F} automatically extends to arbitrary polytopes of dimension $(d+2)(r-1)$. Our final result gives extensions of Theorem~\ref{thm:vK-F} to $p^\ell$-Tverberg families of type $(p,\ldots, p)$. These may likewise be extended to arbitrary polytopes of the appropriate dimensions.

\begin{theorem}
\label{thm:vK-F product}
Let $d\geq 1$ and $\ell\geq 2$  be integers, let $p$ be a prime number, and let $(n_1,\ldots, n_\ell)$ be a $(p^\ell, d; p,\ldots, p)$-Tverberg admissible  $\ell$-tuple guaranteed by Theorems~\ref{thm:main}-~\ref{thm: arbitrary/balanced}. Then any continuous map $f\colon\Delta_{n_1+p-1}\times\cdots\times \Delta_{n_\ell+p-1}\rightarrow \R^d$ admits a $p^\ell$-Tverberg family $\{\sigma_{m_1}^1\times\cdots\times\sigma_{m_\ell}^\ell\mid (m_1,\ldots, m_\ell)\in [p]^\ell\}$ of type $(p,\ldots, p)$ such that $\dim \sigma_i^j\leq \lfloor \frac{n_j}{p}\rfloor$ for all $i\in [p]$ and $j\in [\ell]$.
\end{theorem}

As an example of Theorem~\ref{thm:vK-F product}, the van Kampen--Flores extension of Theorem~\ref{thm: arbitrary/balanced}(b) states that for any continuous map $f\colon \Delta_5\times \Delta_5\rightarrow \R^2$ there is a $4$-Tverberg family $\{\sigma_i\times \tau_j\mid i,j\in [2]\}$ of type $(2,2)$ such that each $\sigma_i$ and $\tau_j$ is at most two-dimensional. Considering multilinear maps, one has that for any set  $X=\{x_{i,j}\mid i,j\in [6]\}$ of 36 points in the plane there are partitions $\{A_1,A_2\}$ and $\{B_1,B_2\}$ of $[6]$ by sets of size three each such that the tensor hulls of the resulting sets $X_{A_i\times B_j}$ completely overlap.

\section{Configuration-Space/Test-Map Scheme for Admissibility}
\label{sec:CS/TM}

The remainder of our paper is devoted to proving the theorems above. In this section we discuss the configuration-space and test-map scheme for all of our results, except that of Theorem~\ref{thm: arbitrary/balanced}(b) and its van Kampen--Flores extension, which we defer to Section~\ref{sec:join}. The remainder of our results are derived in Section~\ref{sec:polynomial proofs} as consequences of a Borsuk--Ulam type statement given in Section~\ref{sec:polynomial criteria} and proven in Section~\ref{sec:B-U Proof}. 

\subsection{Proof Overview}

Our proofs follow the established \emph{configuration-space/test-map} paradigm typically used in topological combinatorics  (see, e.g., ~\cite{Zi17}), so that our results are reduced to corresponding statements concerning equivariant mappings (that is, Borsuk--Ulam type theorems). The construction in our case is a product version of that used in a standard proof of Theorem~\ref{thm:top Tverberg} (see, e.g., ~\cite{Oz87,Vo96, BZ17}). For that result the existence of a $p^\ell$-Tverberg family for a continuous map $f\colon \Delta_n\rightarrow \R^d$ follows from the existence of a zero of a naturally associated $\Z_p^\ell$-equivariant map $F\colon (\Delta_n)^{\times p^\ell}_\Delta\rightarrow U(\Z_p^\ell,d)$ from the $p^\ell$-fold ``deleted product'' of the simplex to a certain $\Z_p^\ell$-module $U(\Z_p^\ell,d)$. Any such map admits a zero when $n\geq (p^\ell-1)(d+1)$ owing to the high connectivity of the domain and the relative simplicity of the mod $p$ cohomology ring of the group $\Z_p^\ell$. Analogously, given a continuous map $f\colon \Delta_{n_1}\times\cdots\times \Delta_{n_k}\rightarrow \R^d$, the existence of a $p^\ell$-Tverberg family of type $(p^{i_1},\ldots, p^{i_k})$ follows from the existence of a zero for any $\Z_p^\ell$-equivariant map $F\colon (\Delta_{n_1})^{\times p^{i_1}}_{\Delta}\times \cdots \times (\Delta_{n_k})^{\times p^{i_k}}_\Delta\rightarrow U(\Z_p^\ell,d)$ from the product of the respective $p^{i_j}$-fold deleted products to the same representation as before. The sum condition $\sum_{j=1}^k n_j\geq t(p^\ell,d;p^{i_1},\ldots, p^{i_k})$ above is a dimensional requirement for this to hold (see Remark~\ref{rem:dimension}). 

While the product domain is no longer highly connected, one is nonetheless guaranteed a zero for the $k$-tuples of our results by a standard equivariant obstruction argument using characteristic classes. Ultimately one is reduced to a coefficient analysis, depending on the choice of integer partition of $\ell$ as well as that of the $n_j$, of a corresponding monomial of a multivariate polynomial determined by the representation $U(\Z_p^\ell,d)$. These polynomials have previously appeared in various contexts in topological combinatorics, including both Tverberg-type and measure equipartition theory. When $i_1=\cdots=i_\ell=1$ is the all-one partition, the resulting Borsuk-Ulam statement given in Proposition~\ref{prop:B-U} below recovers those which are tacitly and in some cases explicitly known, for instance when $p=2$ using ideal-valued cohomological index theory~\cite{FH98} and a $\Z_2$-identification of $(\Delta_n)^{\times 2}_\Delta$ with a $(n-1)$-dimensional sphere equipped with the antipodal action (see ~\cite{Ma08}). To the best of our knowledge, our equivariance result seems to be absent from the literature for all other non-trivial integer partitions of $\ell$, however, and it is in these cases that one may employ a resulting computational flexibility. 

The above construction is enough to establish all of our admissibility results except that of Theorem~\ref{thm: arbitrary/balanced}(b), which can be reduced to an existing Borsuk-Ulam theorem for dihedral group actions~\cite{FS24}.  Theorem~\ref{thm:vK-F product} follows a similar construction to the one outlined above, except now the representation $U(\Z_p^\ell,d)$ is appended by an additional associated representation $U(\Z_p)^{\oplus \ell}$. For this result we note that products of deleted products (or of deleted joins) were previously used in ~\cite{FS24, HMRS26} in obtaining generalizations of a Tverberg-type theorem due to Sarkaria~\cite{Sa90} of relevance to variants of the mass partition problems discussed above.

\subsection{Configuration Space} Our configuration space is a product version of the deleted product construction commonly used in Tverberg-type theory (see, e.g.,~\cite{Oz87, Vo96, MW15, BZ17}). Recall that for any integer $r\geq 2$ and any integer $n\geq r-1$, the deleted $r$-fold product \[(\Delta_n)^{\times r}_\Delta=\{x=(x_1,\ldots, x_r)\in \Delta_n^{\times r}\mid \mathrm{supp}(x_i)\cap \mathrm{supp}(x_j)=\emptyset \,\,\text{for all}\,\, i\neq j\}\] consists of all $r$-tuples $(x_1,\ldots, x_r)$ of points in $\Delta_n$ which have pairwise disjoint support.  

In what follows, it will be crucial to note that $(\Delta_n)^{\times r}_\Delta$ is $(n-r+1)$-dimensional and $(n-r)$-connected (see, e.g.,~\cite{BBS81, BZ17}). Given an abelian group $G$ of order $r$, one has a free $G$-action on $(\Delta_n)^{\times r}_\Delta$ corresponding to addition in the group. Explicitly, fixing an ordering of $G$ gives an indexing of the elements $x=(x_g)_{g\in G}$ of $(\Delta_n)_\Delta^{\times r}$ by the group $G$. One then defines $h\cdot x=(x_{h+g})_{g\in G}$ for each $h\in G$.  Identifying the vertex set of $\Delta_n$ with $[n+1]$ and letting $n$ grow to infinity allows for the union $EG=\cup_{n\geq r-1} (\Delta_n)^{\times r}_\Delta$ to be taken as the total space of the classifying bundle $G\hookrightarrow EG\rightarrow BG=EG/G$ of the group $G$ (see, e.g, ~\cite{Ha00},~\cite{Hu94}). The quotient $(\Delta_n)^{\times r}_\Delta/G$ is then a subcomplex of the classifying space $BG$. 

For ease of notation, for a given prime number $p$ and integer $\ell\geq 2$ we let $r=p^\ell$, and, corresponding to a non-trivial partition $1\leq i_1\leq \cdots \leq i_k\leq \ell$ of $\ell\geq 2$, for each $j\in [k]$ we let $r_j=p^{i_j}$ and $\mathbf{r}=(r_1,\ldots, r_k)$. Supposing that $\mathbf{n}=(n_1,\ldots, n_k)$ is a sequence of integers satisfying $n_j\geq r_j-1$ for all $j\in [k]$, we take as our configuration space the product \[(\Delta_{\mathbf{n}})^{\times \mathbf{r}}_\Delta:=(\Delta_{n_1})^{\times r_1}_\Delta \times\cdots\times (\Delta_{n_k})^{\times r_k}_\Delta\] of the respective deleted $r_j$-fold products of the $\Delta_{n_j}$. Again for notational simplicity we let $G_j=\Z_p^{i_j}$ for each $j\in [k]$, so that $G=\oplus_{j=1}^k G_j$ is their direct sum. One then has a free $G$-action on $(\Delta_{\mathbf{n}})^{\times \mathbf{r}}_\Delta$ given by letting each $G_j$-factor of $G$ act independently on each $(\Delta_{n_j})^{\times r_j}_\Delta$-factor of $(\Delta_{\mathbf{n}})^{\times \mathbf{r}}_\Delta$. 

\subsection{Target Space} Our target space for admissibility is the same as that used in standard proofs of the original Topological Tverberg theorem, as well as in measure equipartition problems. Namely, let $G=\Z_p^\ell$ as above and let $\R[G]=\{\sum_{g\in G} r_g\, g\mid r_g\in \R\,\,\text{for all}\,\, g\in G\}$ be the  right regular representation (group algebra) for $G$, with $G$-action on coordinates again corresponding to that of $G$ on itself by addition. Thus one has $h\cdot \sum_{g\in G} r_g\, g=\sum_{g\in G} r_{h+g}\, g$ for all $h\in G$ and  $\sum_{g\in G} r_g\,g\in \R[G]$. One then considers the $(r-1)$-dimensional subrepresentation $U(G)=\left\{\sum_{g\in G} r_g\, g\in \R[G]\mid \sum_{g\in G}r_g=0\right\}$ consisting of all elements of $\R[G]$ whose coordinates sum to zero. The representation $\R^d[G]=\{\sum_{g\in G} x_g\,g \mid x_g\in \R^d\,\, \text{for all}\,\, g\in G\}$ and the subrepresentation $U(G,d)=\{\sum_{g\in G} x_g\, g\in \R^d[G]\mid \sum_{g\in G} x_g=0\}$ are defined analogously. Observe that the latter may be identified with the $d$-fold sum $U(G)^{\oplus d}$ of $U(G)$ . 

For our van Kampen--Flores results, we shall append to $U(\Z_p^\ell,d)$ the $\ell$-fold sum $U(\Z_p)^{\oplus \ell}$, which we view as a $\Z_p^\ell$-representation by letting each $\Z_p$-factor of $\Z_p^\ell$ act independently on the corresponding $U(\Z_p)$ factor of $U(\Z_p)^{\oplus \ell}$.  As will be needed later, one may then view each component $U(\Z_p)$ as a $\Z_p^\ell$-representation in its own right via projection onto the $j$-th factor, with action given explicitly by $h \cdot \sum_{g\in \Z_p} r_g\,g=\sum_{g\in \Z_p} r_{h_j+g}\, g$ for each $h=(h_1,\ldots, h_\ell)\in \Z_p^\ell$. The $\Z_p^\ell$-representation $U(\Z_p)^{\oplus \ell}$ is then a direct sum of these $\Z_p^\ell$-representations. 

\subsection{Test Map} Let $\mathbf{r}=(r_1,\ldots, r_k)$ and $\mathbf{n}=(n_1,\ldots, n_k)$ be as above and let $f\colon \Delta_{n_1}\times\cdots \times \Delta_{n_k}\rightarrow \R^d$ be a continuous map. We define a continuous $G$-equivariant map \[F\colon (\Delta_{\mathbf{n}})^{\times \mathbf{r}}_\Delta\rightarrow U(G,d)\] as follows. Namely, for each $x=((x_{g_1})_{g_1\in G_1},\ldots, (x_{g_k})_{g_k\in G_k})\in (\Delta_{\mathbf{n}})^{\times \mathbf{r}}_\Delta$ let $x_g=(x_{g_1},\ldots, x_{g_k})$ and let $\avg(f;x)=\frac{1}{|G|}\sum_{g\in G} f(x_g)$ be the average of the images of the $x_g$. We then define 
\[F(x)=\sum_{g\in G} [f(x_g)-\avg(f;x)]\, g.\] By construction, $F(x)=0$ precisely when the collection $\{\supp(x_{g_1})\times \cdots \times \supp(x_{g_k})\mid (g_1,\ldots, g_k)\in G\}$ is an $r$-Tverberg family of type $(r_1,\ldots, r_k)$. It is likewise clear that this map is $G$-equivariant with respect to the action above. Thus one has the following reduction. 

\begin{proposition}
\label{prop:admissibility}
Let $p$ be a prime number, let $d\geq 1$ be an integer, and let $1\leq i_1\leq\cdots\leq i_k$ be a non-trivial partition of an integer $\ell\geq 2$. A $k$-tuple  $(n_1,\ldots, n_k)$ satisfying $n_j\geq (p^{i_j}-1)(d+1)$ for all $j\in [k]$ and $\sum_{j=1}^k n_j=d(p^\ell-1)+\sum_{j=1}^k (p^{i_j}-1)$ is  $(p^\ell,d;p^{i_1},\ldots, p^{i_k})$-Tverberg admissible provided any $\Z_p^\ell$-equivariant map $(\Delta_{(n_1,\ldots,n_k)})^{\times (p^{i_1},\ldots, p^{i_k})}_\Delta\rightarrow U(\Z_p^\ell,d)$ has a zero. \end{proposition}

\begin{remark}
\label{rem:dimension} The condition that $\sum_{j=1}^k n_j\geq t(r,d;r_1,\ldots, r_k)=d(r-1)+\sum_{j=1}^k(r_j-1)$ in Proposition~\ref{prop:admissibility} is necessary on dimensional grounds. Assuming otherwise and letting $S(U(G,d))$ be the representation sphere of $U(G,d)$, one would then have that $\dim((\Delta_{\mathbf{n}})^{\times \mathbf{r}}_\Delta)=\sum_{j=1}^k (n_j-r_j+1)\leq d(r-1)-1=\dim S(U(G,d))$. As the sphere is $\dim S(U(G,d))-1$ connected and the action on $(\Delta_{\mathbf{n}})^{\times \mathbf{r}}_\Delta$ is free, one could then equivariantly extend any $G$-equivariant map defined on the $0$-skeleton of $(\Delta_{\mathbf{n}})^{\times\mathbf{r}}_\Delta$ cell-by-cell to a never-vanishing equivariant map $(\Delta_{\mathbf{n}})^{\times \mathbf{r}}_\Delta\rightarrow S(U(G,d))\hookrightarrow U(G,d)$. 
\end{remark}

For Theorem~\ref{thm:vK-F product} (in which case $k=\ell$), we proceed as in ~\cite{BFZ14, FS24,HMRS26} and augment the given map $F\colon (\Delta_{\mathbf{n}})^{\times \mathbf{r}}_\Delta\rightarrow U(\Z_p^\ell,d)$ by a ``constraint map'' $D\colon (\Delta_{\mathbf{n}})^{\times \mathbf{r}}_\Delta\rightarrow U(\Z_p)^{\oplus \ell}$ which forces $\dim \supp (x_{g_j})\leq \lfloor \frac{n_j}{p}\rfloor$ for each $g_j\in \Z_p$ and all $j\in [\ell]$ whenever $x=((x_{g_1})_{g_1\in \Z_p},\ldots,(x_{g_\ell})_{g_\ell\in\Z_p})\in (\Delta_{\mathbf{n}})^{\times \mathbf{r}}_\Delta$ is a zero of both $D$ and the map $F$ above. To that end, view each $\Delta_{n_j}$ as an abstract simplicial complex. Thus for a face $\sigma\in \Delta_{n_j}$ we have that $|\sigma|$ is the size of its vertex set. We now let $\Sigma_j=\{\sigma\subset \Delta_{n_j}\mid |\sigma|\leq \lfloor\frac{n_j+1}{p}\rfloor\}$ be the abstract simplicial subcomplex of $\Delta_{n_j}$ consisting of all faces with at most $\lfloor\frac{n_j+1}{p}\rfloor$ vertices. We let $\|\Sigma_j\|$ denote its geometric realization, which we view as a subspace of $\R^{n_j}$. As before, for each $x\in (\Delta_{\mathbf{n}})^{\times \mathbf{r}}_\Delta$ let $\avg (d;x,\|\Sigma_j\|)=\frac{1}{p}\sum_{g_j\in \Z_p}\dist(x_{g_j},\|\Sigma_j\|)$ be the average distance of the $p$ points of $x_j=(x_{g_j})_{g_j\in \Z_p}\in (\Delta_{n_j})^{\times p}_\Delta$ to $\|\Sigma_j\|$ and define the map $D_j\colon (\Delta_{\mathbf{n}})^{\times \mathbf{r}}_\Delta\rightarrow U(\Z_p)$ by \[D_j(x)=\sum_{g_j\in \Z_p}\left[\dist (x_{g_j},\|\Sigma_j\|)-\avg(d;x,\|\Sigma_j\|)\right]\,g_j.\] By assumption, the faces of each collection $\{\supp(x_{g_j})\mid g_j\in \Z_p\}$ corresponding to $x\in (\Delta_{\mathbf{n}})^{\times \mathbf{r}}_\Delta$ are pairwise disjoint, and therefore by the pigeon-hole principle for each $j\in [\ell]$ there must be some $g_{j_0}\in \Z_p$ with $\supp(x_{g_{j_0}})\in \Sigma_j$. Thus $\supp(x_{g_j})\in \Sigma_j$ for all $g_j\in \Z_p$ provided $D_j(x)=0$. Letting \[D=(D_1,\ldots, D_\ell)\colon (\Delta_{\mathbf{n}})^{\times \mathbf{r}}_\Delta\rightarrow U(\Z_p)^{\oplus \ell},\] we therefore have that $\dim(\supp(x_{g_j}))\leq \lfloor\frac{n_j+1}{p}\rfloor-1=\lfloor \frac{n_j}{p}\rfloor$ for each $g_j\in \Z_p$ and all $j\in [\ell]$ provided $D(x)=0$. Again it is straightforward to show that this map is $\Z_p^\ell$-equivariant, so as before the existence of the desired $p^\ell$-Tverberg family with dimensionally restricted faces corresponds to a zero of the $\Z_p^\ell$-equivariant map $F\oplus D\colon (\Delta_{\mathbf{n}})^{\times \mathbf{r}}_\Delta\rightarrow U(\Z_p^\ell,d)\oplus U(\Z_p)^{\ell}$. We summarize this as follows.

\begin{proposition}
\label{prop:VK-F product}

Let $p$ be a prime number and let $d\geq 1$ and $\ell\geq 2$ be integers. Suppose that $(n_1,\ldots, n_\ell)$ is an $\ell$-tuple satisfying $n_j\geq (p-1)(d+1)$ for all $j\in [\ell]$ and $\sum_{j=1}^\ell n_j=d(p^\ell-1)+\ell(p-1)$. Then any continuous map $f\colon \Delta_{n_1+p-1}\times \cdots \times \Delta_{n_\ell+p-1}\rightarrow \R^d$ admits a $p^\ell$-Tverberg family $\{\sigma^1_{m_1}\times\cdots\times \sigma^\ell_{m_\ell}\mid (m_1,\ldots, m_\ell)\in [p]^\ell\}$ with $\dim(\sigma^j_{m_j})\leq \lfloor\frac{n_j}{p}\rfloor$ for each $m_j\in [p]$ and each $j\in [\ell]$ provided any $\Z_p^\ell$-equivariant map $(\Delta_{(n_1+p-1,\ldots,n_\ell+p-1)})^{\times (p,\ldots, p)}_\Delta\rightarrow U(\Z_p^\ell,d) \oplus U(\Z_p)^{\oplus \ell}$ has a zero.
\end{proposition}

\section{Polynomial Criteria for Borsuk-Ulam Statements}
\label{sec:polynomial criteria}

Except for Theorem~\ref{thm: arbitrary/balanced}(b) and its van Kampen extension, all of our results follow from polynomial criteria which guarantee zeros for any $\Z_p^\ell$-equivariant map as in Propositions~\ref{prop:admissibility} and~\ref{prop:VK-F product}. These polynomials are directly determined by the decomposition of the representations $U(\Z_p^\ell,d)$ and $U(\Z_p^\ell,d)\oplus U(\Z_p)^{\oplus \ell}$ into real one-dimensional representations when $p=2$, and of a decomposition of related complex representations into one-dimensional complex representations when $p$ is odd. In what follows, it will be convenient to obtain these polynomial conditions as special cases of the Borsuk-Ulam type statement of Proposition~\ref{prop:B-U} below, which holds for arbitrary finite representations.     

We first recall some basic representation theoretic facts (see, e.g., ~\cite{FH04, Ser77}). Namely, the complex irreducible representations of any finite abelian group $G$ are all one-dimensional and are in bijection with the group itself. Explicitly, let $\omega_r=\exp(2\pi i/r)$ be the standard $r$-th root of unity. If $G=\oplus_{j=1}^\ell \Z_{r_j}$ is the direct sum of cyclic groups $\Z_{r_j}$, then each $\alpha=(\alpha_1,\ldots, \alpha_\ell)\in G$ determines the representation $U_\alpha$ given by letting $G$ act on $\C$ by $g\cdot z=\omega_{r_1}^{g_1\alpha_1}\cdots\omega_{r_\ell}^{g_\ell\alpha_\ell}z$ for each $g=(g_1,\ldots, g_\ell)\in G$ and $z\in \C$. For the group $\Z_2^\ell$, one has the analogous situation that its irreducible real representations are all one-dimensional and again indexed by the group. Explicitly, each $\alpha=(\alpha_1,\ldots, \alpha_\ell)\in \Z_2^\ell$ determines the real representation $U_\alpha$ given by letting $\Z_2^\ell$ act on $\R$ by $g\cdot x= (-1)^{g_1\alpha_1+\cdots+g_\ell\alpha_\ell}x$ for each $g=(g_1,\ldots, g_\ell)\in \Z_2^\ell$ and $x\in \R$.

\begin{proposition} 
\label{prop:B-U}
Let $p$ be a prime number, let $1\leq i_1\leq \cdots \leq i_k$ be a non-trivial partition of an integer $\ell\geq 2$, and let $n_1,\ldots, n_k$ be integers such that $n_j\geq p^{i_j}-1$ for all $j\in [k]$. Let $\mathbf{i}_0=0$ and $\mathbf{i}_j=i_1+\cdots+i_j$ for all $j\in [k]$.

\begin{compactenum}[(a)] 

\item Let $p=2$, let $N_j=n_j-2^{i_j}+1$ for all $j\in [k]$, and let $N=\sum_{j=1}^k N_j$. Suppose that $U$ is a real $N$-dimensional representation with $U\cong \oplus_{i=1}^N U_{\alpha_i}$, where $\alpha_i=(\alpha_{i,1},\ldots, \alpha_{i,\ell})\in \Z_2^\ell$ for each $i\in [N]$, and let \[h_U(x_1,\ldots, x_\ell)=\prod_{i=1}^N (\alpha_{i,1}x_1+\cdots+\alpha_{i,\ell}x_\ell)\in \Z_2[x_1,\ldots, x_\ell].\] Assume that there exist non-negative integers $m_1,\ldots, m_\ell$ satisfying  $m_{\mathbf{i}_{j-1}+1}+\cdots+m_{\mathbf{i}_j}=N_j$ for all $j\in[k]$ such that the coefficient of $x_1^{m_1}\cdots x_\ell^{m_\ell}$ in $h_U(x_1,\ldots, x_\ell)$ equals one. Then any $\Z_2^\ell$-equivariant map $F\colon (\Delta_{(n_1,\ldots, n_k)})^{\times (2^{i_1},\ldots, 2^{i_k})}_\Delta\rightarrow U$ has a zero. 

\item Let $p$ be an odd prime and suppose that each $n_j$ is even. Let $N_j=\frac{n_j-p^{i_j}+1}{2}$ for all $j\in [k]$ and let $N=\sum_{j=1}^k N_j$. Suppose that $U$ is a complex $N$-dimensional representation with $U\cong\oplus_{i=1}^N U_{\alpha_i}$, where $\alpha_i=(\alpha_{i,1},\ldots, \alpha_{i,\ell})\in \Z_p^\ell$ for all $i\in [N]$, and let \[h_U(y_1,\ldots, y_\ell)=\prod_{i=1}^N (\alpha_{i,1}y_1+\cdots+\alpha_{i,\ell}y_\ell)\in \Z_p[y_1,\ldots, y_\ell].\]   Assume that there exist non-negative integers $m_1,\ldots, m_\ell$ satisfying  $m_{\mathbf{i}_{j-1}+1}+\cdots+m_{\mathbf{i}_j}=N_j$ for all $j\in[k]$ such that the coefficient of $y_1^{m_1}\cdots y_\ell^{m_\ell}$ in $h_U(y_1,\ldots, y_\ell)$ is non-zero. Then any $\Z_p^\ell$-equivariant map $F\colon (\Delta_{(n_1,\ldots, n_k)})^{\times (p^{i_1},\ldots,p^{i_k})}_\Delta \rightarrow U$ has a zero. 
\end{compactenum}
\end{proposition}

\section{Proof of Proposition~\ref{prop:B-U}}
\label{sec:B-U Proof}

We now prove Proposition~\ref{prop:B-U}. This follows from a standard characteristic class argument applied to the vector bundle given by the Borel construction for the space $(\Delta_{\mathbf{n}})^{\times \mathbf{r}}_\Delta$ and the representation $U$. Throughout, we let $\K=\R$ or $\C$. We refer the reader to~\cite{Hu94, MS74} for background on the theory of vector bundles and characteristic classes used below.

\subsection{Characteristic Classes of Representations.} We begin by recalling the mod~$p$ cohomology of the group $\Z_p^\ell$. As is classically known (see, e.g.,~\cite{At61,Th86}), these are intimately related to characteristic classes of their one-dimensional representations. Namely, the Borel construction applied to a $N$-dimensional $\K$-representation $U$ of a given finite group $G$ produces the $N$-dimensional $\K$-vector bundle \[U\hookrightarrow EU:=EG\times_G U\rightarrow BG\] obtained by modding out the trivial bundle $U\hookrightarrow EG\times U\rightarrow EG$ by the diagonal $G$-action. As in ~\cite{At61}, when $\K=\R$ one defines the Stiefel-Whitney classes $w_i(U)\in H^i(BG;\Z_2)$ and total Stiefel-Whitney class $w(U)=1+\sum_{i=1}^N w_i(U)$ of the representation to be those of the vector bundle $EU$, and likewise for the mod $p$ Chern classes $c_i(U)\in H^{2i}(BG;\Z_p)$ and total Chern class $c(U)$ when $\K=\C$. One then has the following description of $H^*(B\Z_p^\ell;\Z_p)$ in terms of the $\K$-representations $U_{\mathbf{e}_j}$ corresponding to the standard basis vectors $\mathbf{e}_j=(0,\ldots, 1,\ldots, 0)\in\Z_p^\ell$.

\begin{theorem}
\label{thm:cohomology}
Let $\ell\geq 1$ be an integer.  
\begin{compactenum}[(a)]
\item $H^\ast(B\Z_2^\ell;\Z_2)=\Z_2[x_1,\ldots, x_\ell]$, where $x_j=w_1(U_{\mathbf{e}_j})$ for each $j\in [\ell]$.
\item Let $p$ be an odd prime. Then $H^\ast(B\Z_p^\ell;\Z_p)\cong \Lambda[x_1,\ldots, x_\ell]\otimes \Z_p[y_1,\ldots, y_\ell]$, where $\Lambda[\cdot]$ is the exterior algebra, $|x_j|=1$ for all $j\in [\ell]$, and $y_j=c_1(U_{\mathbf{e}_j})$ for all $j\in [\ell]$. 
\end{compactenum}
\end{theorem}

Standard properties of characteristic classes imply that the top Stiefel-Whitney (respectively, mod $p$ Chern) class of the given representation $U$ is the polynomial $h_U(x_1,\ldots, x_\ell)$ (respectively, $h_U(y_1,\ldots, y_\ell)$) of Proposition~\ref{prop:B-U}. Considering the case $\K=\R$, we have by assumption that $U=\oplus_{i=1}^N U_{\alpha_i}$, so by the Whitney sum formula the total Stiefel-Whitney class $w(U)$ is given by $w(U)=\prod_{i=1}^N w(U_{\alpha_i})=\prod_{i=1}^N (1+w_1(U_{\alpha_i}))$. On the other hand, one has vector bundle isomorphisms $EU_{\alpha_i}\cong \otimes_{j=1}^\ell EU_{\mathbf{e}_j}^{\otimes \alpha_{i,j}}$ for each $\alpha_i=(\alpha_{i,1},\ldots, \alpha_{i,\ell})$. Since $w_1(E_1\otimes E_2)=w_1(E_1)+w_1(E_2)$ for any two real line bundles $E_1$ and $E_2$, it follows that $w_1(U_{\alpha_i})=\alpha_{i,1}x_1+\cdots+\alpha_{i,\ell}x_\ell$. Thus $w_N(U)=\prod_{i=1}^N (\alpha_{i,1}x_1+\cdots+\alpha_{i,\ell}x_\ell)=h_U(x_1,\ldots, x_\ell)$ as claimed. When $\K=\C$ the analogous properties of Chern classes of complex vector bundles give $c_N(U)=\prod_{i=1}^N (\alpha_{i,1}y_1+\cdots+\alpha_{i,\ell}y_\ell)=h_U(y_1,\ldots, y_\ell)$ as well.

\subsection{Reduction to Tensor Products}

For the $n_j$ and $N_j$ as given in the statement of Proposition~\ref{prop:B-U} we let $M_j=N_j$ when $\K=\R$, $M_j=2N_j$ when $\K=\C$, and $M=\sum_{j=1}^k M_j$ in both cases. As discussed in Section~\ref{sec:CS/TM}, each $(\Delta_{n_j})^{\times r_j}_\Delta$ is $M_j$-dimensional and $(M_j-1)$-connected. Letting $(EG_j)^{M_j}$ denote the $M_j$-dimensional skeleton of the total space $EG_j$ of the classifying bundle discussed there, by elementary equivariant obstruction theory there is a $G_j$-equivariant map $(EG_j)^{M_j}\rightarrow (\Delta_{n_j})^{\times r_j}_\Delta$ which extends the identity map on  $(\Delta_{n_j})^{\times r_j}_\Delta$. Taking products, one therefore has a $G$-equivariant map from \[E_M(G):=(EG_1)^{M_1}\times\cdots \times (EG_k)^{M_k}\]to $(\Delta_{\mathbf{n}})^{\times \mathbf{r}}_\Delta$, and so Proposition~\ref{prop:B-U} is reduced to showing that any equivariant map $F\colon E_M(G)\rightarrow U$ has a zero. We let $B_M(G):=E_M(G)/G=(BG_1)^{M_1}\times \cdots \times (BG_k)^{M_k}$ be the product of the $M_j$-dimensional skeletons of the respective classifying spaces $BG_j$.

Let $E_M(U)=E_M(G)\times_G U$ denote the total space of the $N$-dimensional $\K$-vector bundle \[U\hookrightarrow E_M(G)\times_G U\rightarrow B_M(G),\] which as before is induced by the diagonal $G$-action on the trivial bundle $U\hookrightarrow E_M(G)\times U\rightarrow E_M(G)$. Observe that $E_M(U)=i^\ast({EU})$ is the pullback of the bundle $EU$ above under the inclusion $i\colon B_M(G)\hookrightarrow BG$, which itself is the product of the inclusions $i_j\colon (BG_j)^{M_j}\hookrightarrow BG_j$. 

Any $G$-equivariant map  $F\colon E_M(G)\rightarrow U$ gives rise to a section $s\colon B_M(G)\rightarrow E_M(U)$ of the vector bundle $E_M(U)$ induced by the section $\widetilde{s}\colon E_M(G)\rightarrow E_M(G)\times U$ of the trivial bundle given by $\widetilde{s}(x)=(x,F(x))$. The map $F$ has a zero if and only if the section $s$ does. On the other hand, the existence of any non-vanishing section of $E_M(U)$ would require the corresponding top-dimensional characteristic classes -- that is, the Stiefel-Whitney class $w_N(E_N(U))\in H^N(B_N(\Z_2^\ell);\Z_2)$ when $U$ is real, and likewise the mod $p$ Chern class $c_N(E_{2N}(U))\in H^{2N}(B_{2N}(\Z_p^\ell);\Z_p)$ when $U$ is complex -- to be zero. Thus it suffices to show that these respective classes are in fact non-zero. To do so, one may use the naturality of characteristic classes as applied to the inclusion $i\colon B_M(G)\rightarrow BG$ above, which gives that $w_N(E_N(U))=i^\ast(w_N(U))$ in the real case and likewise $c_N(E_{2N}(U))=i^\ast(c_N(U))$ in the complex setting. 

We now use a  K\"unneth formula argument. As coefficients are in a field we have $H^\ast(B_M(G);\Z_p)\cong \otimes_{j=1}^k H^\ast((BG_j)^{M_j};\Z_p)$ and $H^\ast(BG;\Z_p)\cong \otimes_{j=1}^kH^\ast(BG_j;\Z_p)$ (see, e.g.,~\cite{Ha00}). By the naturality of the K\"unneth formula, the homomorphism $i^\ast\colon H^\ast(BG;\Z_p)\rightarrow H^\ast(B_M(G);\Z_p)$ may be identified with the tensor product $\otimes_{j=1}^k H^\ast(BG_j;\Z_p)\rightarrow \otimes_{j=1}^k H^*((BG_j)^{M_j};\Z_p)$ of the homomorphisms $i_j^\ast\colon H^\ast (BG_j;\Z_p)\rightarrow H^\ast((BG_j)^{M_j};\Z_p)$. On the other hand, the long exact sequence of the pair $(BG_j,(BG_j)^{M_j})$ implies that each homomorphism $i_j^\ast\colon H^{M_j}(BG_j;\Z_p)\rightarrow H^{M_j}((BG_j)^{M_j};\Z_p)$ is injective. As these homomorphisms are of vector spaces, the corresponding tensor map

\[\otimes_{j=1}^k H^{M_j}(BG_j;\Z_p)\rightarrow \otimes_{j=1}^k H^{M_j}((BG_j)^{M_j};\Z_p)\] is injective as well. 

To complete the proof, consider the integers $m_1,\ldots, m_\ell$ in the statement of Proposition~\ref{prop:B-U}. Recall that $\mathbf{i}_0=0$ and  $\mathbf{i}_j=i_1+\cdots+i_j$ for all $j\in [k]$. Considering the real case first, we have $w_N(U)=h_U(x_1,\ldots,x_\ell)\in \Z_2[x_1,\ldots, x_\ell]\cong \otimes_{j=1}^k \Z_2[x_{\mathbf{i}_{j-1}+1},\ldots, x_{\mathbf{i}_j}]$. By assumption, for each $j\in [k]$ we have $m_{\mathbf{i}_{j-1}+1}+\cdots +m_{\mathbf{i}_j}=N_j$ and so that $i_j^\ast(x_{\mathbf{i}_{j-1}+1}^{m_{\mathbf{i}_{j-1}+1}}\cdots x_{\mathbf{i}_j}^{m_{\mathbf{i}_j}})\neq 0$. It then follows from the injectivity of the tensor map above that $i^\ast(x_1^{m_1}\cdots x_\ell^{m_\ell})$ is non-zero. Since the coefficient of $x_1^{m_1}\cdots x_\ell^{m_\ell}$ in $h_U(x_1,\ldots, x_\ell)$ is equal to one by assumption, we conclude that $w_N(E_N(U))=i^*(w_N(U))$ is non-zero. This completes the proof of Proposition~\ref{prop:B-U}(a).

The argument when $p$ is odd is identical; one only needs to replace the variables $x_j$ by $y_j$. As before, the injectivity of the tensor map gives  $i^\ast(y_1^{m_1}\cdots y_\ell^{m_\ell})\neq 0$, and since the coefficient of $y_1^{m_1}\cdots y_\ell^{m_\ell}$ in $c_N(U)=h_U(y_1,\ldots, y_\ell)$ is non-zero by assumption we have $c_N(E_{2N}(U))=i^\ast (c_N(U))\neq 0$.  

\section{Proofs of Theorems~\ref{thm:main}--~\ref{thm: arbitrary/balanced} and Theorem~\ref{thm:vK-F product}}
\label{sec:polynomial proofs}

 We now specialize Proposition~\ref{prop:B-U} to the polynomials corresponding to the representations associated to Propositions~\ref{prop:admissibility} and ~\ref{prop:VK-F product}. 

\subsection{Polynomials for the case $p=2$} The polynomial $h_{U(\Z_2^\ell,d)}(x_1,\ldots, x_\ell)$ is well-known, for instance from its use (see, e.g, ~\cite{MLVZ06,BCCD23,FS24}) in the context of the Gr\"unbaum--Hadwiger--Ramos hyperplane equipartition problem above and its relatives. Nonetheless, we  provide a brief derivation for the sake of completeness. Namely, as the real regular representation $\R[\Z_2^\ell]$ decomposes as the sum of all real irreducible $\Z_2^\ell$-representations (see, e.g., ~\cite{Ser77,FH04}), one has that $U(\Z_2^\ell)\cong \oplus_{\alpha\in \Z_2^\ell\setminus\{0\}}U_\alpha$ and therefore that $U(\Z_2^\ell,d)\cong \oplus_{\alpha\neq 0} U_\alpha^{\oplus d}$. Owing to the explicit calculation $\prod_{(\alpha_1,\ldots, \alpha_\ell)\neq 0} (\alpha_1 x_1+\cdots+\alpha_\ell x_\ell)=\sum_{\sigma\in\mathfrak{S}_\ell}x_{\sigma(1)}^{2^{\ell-1}}x_{\sigma(2)}^{2^{\ell-2}}\cdots x_{\sigma(\ell)}$ given in ~\cite{Wi83}, one has \begin{equation} 
\label{eqn:p=2} h_{U(\Z_2^\ell,d)}(x_1,\ldots, x_\ell)=\left(\sum_{\sigma\in\mathfrak{S}_\ell}x_{\sigma(1)}^{2^{\ell-1}}\cdots x_{\sigma(\ell)}\right)^d.\end{equation} Viewing $U(\Z_2)$ as a $\Z_2^\ell$-representation via projection onto the $j$-th coordinate as in Section~\ref{sec:CS/TM}, one has that $U(\Z_2)\cong U_{\mathbf{e}_j}$ and therefore that $U(\Z_2)^{\oplus \ell}\cong \oplus_{j=1}^\ell U_{\mathbf{e}_j}$. Thus $h_{U(\Z_2)^{\oplus \ell}}(x_1,\ldots, x_\ell)=x_1\cdots x_\ell$ and so the relevant polynomial for Proposition~\ref{prop:VK-F product} is  \begin{equation} 
\label{eqn:p=2 VK} h(x_1,\ldots, x_\ell)=x_1\cdots x_\ell\left(\sum_{\sigma\in\mathfrak{S}_\ell}x_{\sigma(1)}^{2^{\ell-1}}\cdots x_{\sigma(\ell)}\right)^d.\end{equation} 

\subsection{Polynomials for odd $p$} One has analogous decompositions of the representations corresponding to Propositions~\ref{prop:admissibility} and ~\ref{prop:VK-F product} when $p$ is an odd prime, as follows from identifying $U(\Z_p^\ell,d)$ and $U(\Z_p)^{\oplus \ell}$ with certain complex representations. The former was done for instance in~\cite{Si15} and the latter in~\cite{HMRS26} in the context of certain complex relatives of the Gr\"unbaum--Hadwiger--Ramos hyperplane equipartition problem. Again in the interest of self-containment we explain their derivations here. So let $\C[\Z_p^\ell]=\{\sum_{g\in \Z_p^\ell}z_g\, g\mid z_g\in \C\}$ be the complex right regular representation of $\Z_p^\ell$ and let $U_{\mathbb{C}}(\Z_p^\ell)=\{\sum_{g\in \Z_p^\ell} z_g\, g\in \C[\Z_p^\ell]\mid \sum_{g\in \Z_p^\ell} z_g=0\}$ be the subrepresentation of those elements whose coefficients sum to zero. Again one has the decomposition $\C[\Z_p^\ell]\cong \oplus_{\alpha\in \Z_p^\ell} U_\alpha$ of complex representations (see ~\cite{Ser77, FH04}), and so that $U_\C(\Z_p^\ell)\cong \oplus_{\alpha\neq 0} U_\alpha$ as before. Now, one may view any finite dimensional complex representation $W$ as a real representation $W_\R$ of twice the dimension. In our case, complex conjugation shows that $(U_\alpha)_\R$ and $(U_{-\alpha})_\R$ are isomorphic for any $\alpha\in \Z_p^\ell$, while $(U_\C(\Z_p^\ell))_\R\cong U(\Z_p^\ell)^{\oplus 2}$. Letting $A$ be the subset of $\Z_p^\ell\setminus\{0\}$ consisting of all $\alpha\neq 0$ whose last non-zero coordinate lies in $[(p-1)/2]$, it follows that \[U(\Z_p^\ell)^{\oplus 2}\cong (U_\C(\Z_p^\ell))_\R\cong (\oplus_{\alpha\neq 0} U_\alpha)_\R\cong (\oplus_{\alpha\in A} U_\alpha)^{\oplus 2}.\] Thus the real representation $U(\Z_p^\ell)$ may be identified with the complex representation $\oplus_{\alpha\in A}U_\alpha$, and so $U(\Z_p^\ell,d)$ may be identified with $U_A(\Z_p^\ell,d):=\oplus_{\alpha\in A} U_\alpha^{\oplus d}$. 

For each $k\in [(p-1)/2]$, let $A_k\subset A$ be the subset consisting of those $\alpha\neq 0$ whose last non-zero coordinate is equal to $k$. Thus $A=\cup_{k=1}^{(p-1)/2} A_k$, and we have $\prod_{(\alpha_1,\ldots, \alpha_\ell)\in A_k}(\alpha_1y_1+\cdots+\alpha_\ell y_\ell)=k^{(p^\ell-1)/(p-1)}\prod_{(\alpha_1,\ldots, \alpha_\ell)\in A_1}(\alpha_1y_1+\cdots+\alpha_\ell y_\ell)$ for each $k\in [(p-1)/2]$. As before, one has the explicit formula $\prod_{\alpha\in A_1}(\alpha_1y_1+\cdots+\alpha_\ell y_\ell)=\sum_{\sigma\in \mathfrak{S}_\ell}\sgn(\sigma)y_{\sigma(1)}^{p^{\ell-1}}\cdots y_{\sigma(\ell)}$ (see ~\cite{Wi83}). Thus $h_{U_A(\Z_p^\ell,d)}(y_1,\ldots, y_\ell)$ is a non-zero constant multiple of \begin{equation}\label{eqn: p odd} h(y_1,\ldots, y_\ell):= \left(\sum_{\sigma\in \mathfrak{S}_\ell}\sgn(\sigma)y_{\sigma(1)}^{p^{\ell-1}}\cdots y_{\sigma(\ell)}\right)^{d(p-1)/2},\end{equation} where the constant is explicitly given by $[((p-1)/2)!]^{d(p^\ell-1)/(p-1)}\in \Z_p$.

In a similar fashion, the real-$\Z_p^\ell$ representation $U(\Z_p)$ given by projection onto the $j$-th coordinate discussed in Section~\ref{sec:CS/TM} can be identified with the complex representation $\oplus_{a=1}^{(p-1)/2} U_{a\mathbf{e}_j}$, and so $U(\Z_p)^{\oplus \ell}$ identifies with $\oplus_{j=1}^\ell \bigoplus_{a=1}^{(p-1)/2} U_{a\mathbf{e}_j}$. Thus the relevant polynomial for the odd prime cases of Proposition~\ref{prop:VK-F product} is a non-zero constant multiple of 
\begin{equation}\label{eqn:p odd VK} h(y_1,\ldots, y_\ell)=y_1^{(p-1)/2}\cdots y_\ell^{(p-1)/2}\left(\sum_{\sigma\in \mathfrak{S}_\ell}\sgn(\sigma)y_{\sigma(1)}^{p^{\ell-1}}\cdots y_{\sigma(\ell)}\right)^{d(p-1)/2}.
\end{equation} Here one may check that the explicit constant is now given by $[((p-1)/2)!]^{\ell+d(p^\ell-1)/(p-1)}$.\\

Using the polynomial calculations above we now give proofs of all of our results except Theorem~\ref{thm: arbitrary/balanced}(b) and its van Kampen--Flores extension.

\subsection{Proof of Theorem~\ref{thm:main}} We begin with a proof of our main result, in which case the coefficient analysis is particularly simple. 

\begin{proof}[Proof of Theorem~\ref{thm:main}] Let $1\leq i_1\leq\cdots\leq i_k$ be a non-trivial partition of $\ell$ and suppose that $C_1,\ldots, C_k$ are disjoint subsets of $[\ell]$ of size $i_j$ each. We let $n_j=d(p-1)\cdot \sum_{a\in C_j} p^{a-1}+p^{i_j}-1$.

When $p=2$, consider $h_{U(\Z_2^\ell,d)}(x_1,\ldots, x_\ell)=(\sum_{\sigma\in \mathfrak{S}_\ell}x_{\sigma(1)}^{2^{\ell-1}}\cdots x_{\sigma(\ell)})^d$. The coefficient of the monomial $x_1^{d\cdot2^{{\varphi(1)-1}}}\cdots x_{\ell}^{d\cdot 2^{\varphi(\ell)-1}}$ equals one for any $\varphi\in \mathfrak{S}_\ell$. To complete the proof we only need let $m_i=d\cdot 2^{\varphi(i)-1}$ for all $i\in [\ell]$, where $\varphi$ is any permutation which satisfies $\{\varphi(\mathbf{i}_{j-1}+1),\ldots, \varphi(\mathbf{i}_j)\}=C_j$ for all $j\in [k]$. We then have $m_{\mathbf{i}_{j-1}+1}+\cdots+m_{\mathbf{i}_j}=d\cdot \sum_{a\in C_j}2^{a-1}=n_j-2^{i_j}+1$ for all $j\in [k]$, so $(n_1,\ldots, n_k)$ is $(2^\ell,d;2^{i_1},\ldots, 2^{i_k})$-Tverberg admissible by Propositions~\ref{prop:B-U}(a) and~\ref{prop:admissibility}.

For odd primes $p$ we have as above that the monomial $y_1^{d(p-1)/2\cdot p^{{\varphi(1)-1}}}\cdots y_{\ell}^{d(p-1)/2\cdot p^{\varphi(\ell)-1}}$ of $h(y_1,\ldots, y_\ell)=(\sum_{\sigma\in \mathfrak{S}_\ell}\sgn(\sigma)y_{\sigma(1)}^{p^{\ell-1}}\cdots y_{\sigma(\ell)})^{d(p-1)/2}$ has coefficient equal to $\pm 1$ for any $\varphi\in \mathfrak{S}_\ell$, and therefore that the coefficient of this monomial in $h_{U_A(\Z_p^\ell,d)}(y_1,\ldots, y_\ell)$ is non-zero. Choosing $\varphi\in \mathfrak{S}_\ell$ as before, we let $m_i=\frac{d(p-1)}{2}\cdot p^{\varphi(i)-1}$ for all $i\in [\ell]$. Then $m_{\mathbf{i}_{j-1}+1}+\cdots+m_{\mathbf{i}_j}=(n_j-p^{i_j}+1)/2$ for all $j\in [k]$, so  $(n_1,\ldots, n_k)$ is $(p^\ell,d;p^{i_1},\ldots, p^{i_k})$-Tverberg admissible by Propositions~\ref{prop:B-U}(b) and Proposition~\ref{prop:admissibility}.
\end{proof}

\subsection{Proof of Theorem~\ref{thm:p=2 general}} We now prove our other admissibility result for arbitrary non-trivial integer partitions, which holds when $p=2$. In this case our analysis of $h_{U(\Z_2^\ell,d)}(x_1,\ldots, x_\ell)$ is based on the decomposition of $d$ into the form $d=2^{s_1}+t_1$ with $s_1\geq 0$ and $0\leq t_1<2^{s_1}$ and follows that of ~\cite{Si19} used in the context of constrained versions of the Gr\"unbaum--Hadwiger--Ramos problem.

We first need the following lemma. 

\begin{lemma}
\label{lem:decreasing} The sequence $u_1,\ldots, u_\ell$ in the statement of Theorem~\ref{thm:p=2 general} is strictly decreasing. 
\end{lemma}

\begin{proof} We recall that the numbers $s_2,\ldots s_{\ell-1}$ and $t_2,\ldots, t_{\ell-1}$ are defined recursively for $1\leq j \leq \ell-2$ by $2^{s_{j+1}}+t_{j+1}=2^{s_j}+2t_j$, where $s_{j+1}\geq 0$ and $0\leq t_{j+1}<2^{s_{j+1}}$. Correspondingly, we have $u_j=2^{s_j+\ell-j}+t_j$ for each $1\leq j \leq \ell-1$ and $u_\ell=2^{s_{\ell-1}}+2t_{\ell-1}$. 

As $u_{\ell-1}=2^{s_{\ell-1}+1}+t_{\ell-1}=2^{s_{\ell-1}}+2^{s_{\ell-1}}+t_{\ell-1}>2^{s_{\ell-1}}+2t_{\ell-1}=u_\ell,$ we are left to show that $u_j>u_{j+1}$ for all $1\leq j \leq \ell-2$. Here there are two cases, namely whether $0\leq t_j<2^{s_j-1}$, or else that $t_j= 2^{s_j-1}+a_j$ for some integer $0\leq a_j<2^{s_j-1}$. In the first case one has $s_{j+1}=s_j$ and $t_{j+1}=2t_j$, and therefore that
\[u_j=2^{s_j+\ell-j}+t_j=2^{s_j+\ell-j-1}+2^{s_j+\ell-j-1}+t_j> 2^{s_j+\ell-j-1}+2^{s_j}+t_j>2^{s_j+\ell-j-1}+2t_j=u_{j+1}.\] In the second case we have $s_{j+1}=s_j+1$ and $t_{j+1}=2a_j<t_j$. Thus \[u_j=2^{s_j+\ell-j}+t_j=2^{s_{j+1}+\ell-j-1}+t_j>2^{s_{j+1}+\ell-j-1}+t_{j+1}=u_{j+1},\] which completes the proof. \end{proof}

\begin{proof}[Proof of Theorem~\ref{thm:p=2 general}] We have \begin{align*} h_{U(\Z_2^\ell,d)}(x_1,\ldots, x_\ell)=& (\sum_{\sigma\in \mathfrak{S}_\ell}x_{\sigma(1)}^{2^{\ell-1}}\cdots x_{\sigma(\ell)})^{2^{s_1}+t_1}\\
=& (\sum_{\sigma\in \mathfrak{S}_\ell}x_{\sigma(1)}^{2^{s_1+\ell-1}}\cdots x_{\sigma(\ell)}^{2^{s_1}})\cdot (\sum_{\tau\in \mathfrak{S}_\ell}x_{\tau(1)}^{2^{\ell-1}}\cdots x_{\tau(\ell)})^{t_1}.
\end{align*}

We first show that the coefficient of $x_1^{u_1}\cdots x_\ell^{u_\ell}$ in $h_{U(\Z_2^\ell,d)}(x_1,\ldots, x_\ell)$ is one by showing that $h_{U(\Z_2^\ell,d)}(x_1,\ldots, x_\ell)=x_1^{u_1}\cdots x_\ell^{u_\ell}$ in the truncated polynomial ring $\Z_2[x_1,\ldots, x_\ell]/(x_1^{u_1+1},\ldots, x_\ell^{u_\ell+1})$. The same argument then shows that the coefficient of  $x_1^{u_{\varphi(1)}}\cdots x_\ell^{u_{\varphi(\ell)}}$ in $h_{U(\Z_2^\ell,d)}(x_1,\ldots, x_\ell)$ is also one for any $\varphi\in \mathfrak{S}_\ell$. 

In the truncated ring we have $x_1^{u_1+1}=0$, and also $x_j^{u_1}=0$ for all $2\leq j\leq \ell$ by Lemma~\ref{lem:decreasing}. Letting $\mathfrak{S}_{\ell-i}$ denote the symmetric group on $\{i+1,\ldots, \ell\}$ for any $i\in [\ell]$, it follows that 

\begin{align*} h_{U(\Z_2^\ell,d)}(x_1,\ldots, x_\ell)&= x_1^{u_1}\cdot (\sum_{\sigma\in \mathfrak{S}_{\ell-1}}x_{\sigma(2)}^{2^{s_1+\ell-2}}\cdots x_{\sigma(\ell)}^{2^{s_1}})\cdot (\sum_{\tau\in \mathfrak{S}_{\ell-1}}x_{\tau(2)}^{2^{\ell-1}}\cdots x^2_{\tau(\ell)})^{t_1}.\end{align*} 

If $\ell=2$ we have $h_{U(\Z_2^2,d)}(x_1,x_2)=x_1^{u_1}x_2^{u_2}$, and otherwise we have 
\begin{align*} h_{U(\Z_2^\ell,d)}(x_1,\ldots, x_\ell)&=  x_1^{u_1}\cdot (\sum_{\sigma\in \mathfrak{S}_{\ell-1}}x_{\sigma(2)}^{2^{\ell-2}}\cdots x_{\sigma(\ell)})^{2^{s_1}+2t_1}\\
&=x_1^{u_1}\cdot (\sum_{\sigma\in \mathfrak{S}_{\ell-1}}x_{\sigma(2)}^{2^{\ell-2}}\cdots x_{\sigma(\ell)})^{2^{s_2}+t_2}\\
&= x_1^{u_1}\cdot (\sum_{\sigma\in \mathfrak{S}_{\ell-1}}x_{\sigma(2)}^{2^{s_2+\ell-2}}\cdots x_{\sigma(\ell)}^{2^{s_2}})\cdot (\sum_{\tau\in \mathfrak{S}_{\ell-1}}x_{\tau(2)}^{2^{\ell-2}}\cdots x_{\tau(\ell)})^{t_2}.
\end{align*}

We now proceed inductively. Again $x_2^{u_2+1}=0$ and $x_j^{u_2}=0$ for all $3\leq j\leq \ell$ by Lemma~\ref{lem:decreasing}, so the same computation as above gives \[h_{U(\Z_2^\ell,d)}(x_1,\ldots, x_\ell)=x_1^{u_1}x_2^{u_2}\cdot (\sum_{\sigma\in \mathfrak{S}_{\ell-2}}x_{\sigma(3)}^{2^{s_2+\ell-3}}\cdots x_{\sigma(\ell)}^{2^{s_2}})\cdot (\sum_{\tau\in \mathfrak{S}_{\ell-2}}x_{\tau(3)}^{2^{\ell-2}}\cdots x^2_{\tau(\ell)})^{t_2}.\]
Continuing gives $h_{U(\Z_2^\ell,d)}(x_1,\ldots, x_\ell)=x_1^{u_1}\cdots x_\ell^{u_\ell}$ in $\Z_2[x_1,\ldots, x_\ell]/(x_1^{u_1+1},\ldots, x_\ell^{u_\ell+1})$. Note that this calculation also shows that $u_1+\cdots+u_\ell=d(2^\ell-1)$.

We now consider the given non-trivial partition $1\leq i_1\leq \cdots\leq i_k$ of $\ell$ and disjoint subsets $C_1,\ldots, C_k$ of $[\ell]$ of size $i_j$ each. Let $n_j=\sum_{a\in C_j} u_a+2^{i_j}-1$ for all $j\in [k]$. As in the proof of Theorem~\ref{thm:main}, we let $\varphi\in \mathfrak{S}_\ell$ be such that $\{\varphi(\mathbf{i}_{j-1}+1),\ldots, \varphi(\mathbf{i}_j)\}=C_j$ for all $j\in [k]$ and we define $m_i=u_{\varphi(i)}$ for all $i\in [\ell]$. As there, we have that $m_{\mathbf{i}_{j-1}+1}+\cdots+m_{\mathbf{i}_j}=N_j:=n_j-2^{i_j}+1$ for all $j\in [k]$. On the other hand, the polynomial calculation above establishes that $\sum_{j=1}^k n_j=d(2^\ell-1)+\sum_{j=1}^k (2^{i_j}-1)=t(2^\ell,d;2^{i_1},\ldots, 2^{i_k})$. Since the coefficient of $x_1^{m_1}\cdots x_\ell^{m_\ell}$ in $h_{U(\Z_2^\ell,d)}(x_1,\ldots, x_\ell)$ is equal to one, we conclude from Propositions~\ref{prop:B-U}(a) and~\ref{prop:admissibility} that the $k$-tuple $(n_1,\ldots, n_k)$ is $(2^\ell,d;2^{i_1},\ldots, 2^{i_k})$-Tverberg admissible. \end{proof}

\subsection{Proof of Theorem~\ref{thm: arbitrary/balanced}(a)(c)(d)}

We now prove our admissibility results for $p^{2\ell}$-Tverberg tuples of type $(p^\ell,p^\ell)$, except that given by Theorem~\ref{thm: arbitrary/balanced}(b). Parts (a) and (c) of Theorem~\ref{thm: arbitrary/balanced} follow as immediate consequences of a simple binomial coefficient test for $(p^2,d;p,p)$-admissibility:

\begin{theorem}
\label{thm:binomial} Let $d\geq 1$ be an integer.
\begin{compactenum}[(a)]
\item Let $i\in\{0,\ldots, d\}$, and let $n_1=d+i+1$ and $n_2=2d-i+1$. If $\binom{d}{i}$ is odd then any continuous map $f\colon \Delta_{n_1}\times\Delta_{n_2}\rightarrow \R^d$ admits a $4$-Tverberg family of type $(2,2)$. 

\item Let $p$ be an odd prime, let $i\in\{0,\ldots, d(p-1)/2\}$, and let $n_1=(p-1)(d+2i+1)$ and $n_2=(p-1)(pd-2i+1)$. If $\binom{d(p-1)/2}{i}\not\equiv 0 \pmod{p}$ then any continuous map $f\colon \Delta_{n_1}\times \Delta_{n_2}\rightarrow \R^d$ admits a $p^2$-Tverberg family of type $(p,p)$.
\end{compactenum}
\end{theorem}

\begin{proof}[Proof of Theorem~\ref{thm:binomial}] We first consider $4$-Tverberg families of type $(2,2)$. For given $i\in\{0,\ldots, d\}$ we let  $n_1=d+i+1$ and $n_2=2d-i+1$. We have $h_{U(\Z_2^2,d)}(x_1,x_2)=(x_1^2x_2+x_2^2x_1)^d=\sum_{k=0}^d\binom{d}{k} x_1^{d+k}x_2^{2d-k}$, so Propositions~\ref{prop:B-U}(a) and ~\ref{prop:admissibility} show that $(n_1,n_2)$ is $(4,d;2,2)$-Tverberg admissible provided $\binom{d}{i}=1$ in $\Z_2$. 

The proof when $p$ is an odd prime is analogous. Let $a=d(p-1)/2$. We have that $h_{U_A(\Z_p^2,d)}(y_1,y_2)$ is a non-zero multiple of 
$h(y_1,y_2)=(y_1^py_2-y_2^py_1)^a=\sum_{k=0}^a\binom{a}{k}(-1)^{a-k}y_1^{a+k(p-1)}y_2^{pa-k(p-1)}$. Letting $i\in\{0,\ldots,a\}$, an application of Proposition~\ref{prop:B-U}(b) together with Proposition~ \ref{prop:admissibility} then gives the admissibility of $(n_1,n_2)$ with $n_1=2a+2i(p-1)+p-1=(p-1)(d+2i+1)$ and $n_2=2pa-2i(p-1)+p-1=(p-1)(pd-2i+1)$, provided $\binom{d(p-1)/2}{i}\neq 0$ in $\Z_p$.
\end{proof}

We now prove Theorem~\ref{thm: arbitrary/balanced}(a)(c)(d). 

\begin{proof}[Proof of Theorem~\ref{thm: arbitrary/balanced}(a)(c)(d)]
For parts (a) and (c), recall that  if $a=a_jp^j+\cdots+a_1p+a_0$ and $b=b_jp^j+\cdots+b_1p+b_0$ are the $p$-ary expansions of two integers $a,b\geq 1$, then Lucas's theorem (see, e.g, ~\cite{Fi47}) guarantees that $\binom{a}{b}\equiv \prod_{k=0}^j\binom{a_k}{b_k}\pmod{p}$. Now suppose that $d=2^{s+1}-1$ for some integer $s\geq 0$. Its binary expansion is then $d=2^s+2^{s-1}+\cdots+1$, so given any $i=i_s 2^s+\cdots +i_0\in \{0,\ldots, d\}$ we have $\binom{d}{i}=\prod_{k=0}^s \binom{1}{i_k}\neq 0$ in $\Z_2$. Thus $(d+i+1,2d-i+1)$ is $(4,d;2,2)$-Tverberg admissible by Theorem~\ref{thm:binomial}. For Theorem~\ref{thm: arbitrary/balanced}(c), suppose similarly that $d(p-1)/2=a_jp^j+\cdots+a_1p+a_0$, where each coefficient is even (and so $d(p-1)/2$ is itself even). To show $(n,n)$-admissibility with $n=d(p^2-1)/2+p-1$, by Theorem~\ref{thm:binomial} it is enough to show that $\binom{d(p-1)/2}{i}$ is non-zero with $i=d(p-1)/4$. But  $\binom{d(p-1)/2}{d(p-1)/4}=\prod_{k=0}^j \binom{a_k}{a_k/2}\neq 0$, so $(n,n)$ is $(p^2,d;p,p)$-Tverberg admissible. 

Finally, for part (d) we have by assumption that $d(p-1)/2=2p^s$ where $s\geq 0$ is an integer. Thus either $p=3$ and $d=2\cdot 3^s$, or else $p=5$ and $d=5^s$, and we observe that $d(p^{2\ell}-1)/2$ is even in both cases (if $p=5$, then $p\equiv 1\pmod{4})$. We have $n_1=n_2=d(p^{2\ell}-1)/2+p^\ell-1$. To use Proposition~\ref{prop:B-U}(b), we let $N_1=N_2=d(p^{2\ell}-1)/4$. We only need to show that the coefficient $c$ of the monomial \[y_1^{p^s\cdot (p^{2\ell-1}+1)}\cdots y_\ell^{p^s\cdot (p^\ell+p^{\ell-1})} y_{\ell+1}^{p^s\cdot (p^\ell+p^{\ell-1})}\cdots y_{2\ell}^{p^s\cdot (p^{2\ell-1}+1)}\] in 
\begin{align*} h(y_1,\ldots, y_{2\ell})&=(\sum_{\sigma\in \mathfrak{S}_{2\ell}}\sgn(\sigma)y_{\sigma(1)}^{p^{2\ell-1}}\cdots y_{\sigma(2\ell)})^{2p^s}\\&=(\sum_{\sigma\in \mathfrak{S}_{2\ell}} \sgn(\sigma)y_{\sigma(1)}^{p^s\cdot p^{2\ell-1}}\cdots y_{\sigma(2\ell)}^{p^s})\cdot (\sum_{\tau\in \mathfrak{S}_{2\ell}} \sgn(\tau)y_{\tau(1)}^{p^s\cdot p^{2\ell-1}}\cdots y_{\tau(2\ell)}^{p^s})\end{align*} is non-zero. To that end, for each $i\in [2\ell]$ let $j=2\ell-i+1$. By uniqueness of $p$-ary expansions, the monomial considered occurs in $h(y_1,\ldots, y_{2\ell})$ precisely for those $\sigma,\tau\in \mathfrak{S}_{2\ell}$ such that for every $i\in [2\ell]$ one has that (1) $\sigma(i)=i=\tau(j)$ and $\sigma(j)=j=\tau(i)$, or else that (2) $\sigma(i)=j=\tau(j)$ and $\sigma(j)=i=\tau(i)$. Thus the $\sigma$ and $\tau$ which produce the desired monomial satisfy, for every $i\in [2\ell]$, that $\sigma$ is the identity on $\{i,j\}$ while $\tau$ acts as a transposition on $\{i,j\}$, or else vice versa. Hence there are exactly $2^\ell$ possibilities for $\sigma$ and $\tau$. Moreover, for each such pair one has $\sgn(\tau)=(-1)^\ell\sgn(\sigma)$. Thus $\sgn(\sigma)\sgn(\tau)=(-1)^\ell$ in each case, and so  $c=(-1)^\ell2^\ell\neq 0$ in $\Z_p$. This completes the proof of Theorem~\ref{thm: arbitrary/balanced}(d). 
\end{proof}

We conclude this section with a proof of the van Kampen--Flores extensions of all of our results, except that of Theorem~\ref{thm: arbitrary/balanced}(b).

\begin{proof}[Proof of Theorem~\ref{thm:vK-F product}]

The cases of Theorem~\ref{thm:vK-F product}, except that corresponding to Theorem~\ref{thm: arbitrary/balanced}(b), all follow easily from the calculations of this section. When $p=2$, we have proved $(2^\ell,d;2,\ldots, 2)$-Tverberg admissibility of an $\ell$-tuple $(n_1,\ldots, n_\ell)$ owing to the fact that the coefficient of $x_1^{n_1-1}\cdots x_\ell^{n_\ell-1}$ in $h_{U(\Z_2^\ell,d)}(x_1,\ldots, x_\ell)$ is equal to one. The coefficient of $x_1^{n_1}\cdots x_\ell^{n_\ell}$ in $x_1\cdots x_\ell h_{U(\Z_2^\ell,d)}(x_1,\ldots, x_\ell)$ is therefore one as well. So by Proposition~\ref{prop:B-U}(a) and Proposition~\ref{prop:VK-F product} one has the desired $2^\ell$-Tverberg family for any map $f\colon \Delta_{n_1+1}\times\cdots\times \Delta_{n_\ell+1}\rightarrow \R^d$.

Again the proof for odd primes $p$ is analogous. Now we have the $(p^\ell,d;p,\ldots,p)$-Tverberg admissibility of $(n_1,\ldots, n_\ell)$, where the $n_j$ are all even, by having shown that the coefficient of $y_1^{N_1}\cdots y_\ell^{N_\ell}$ in $h_{U_A(\Z_p^\ell,d)}(y_1,\ldots, y_\ell)$ is non-zero, where we recall that $N_j=\frac{n_j-p+1}{2}$ for all $j\in [\ell]$. But this coefficient is the same as that of $y_1^{N_1+(p-1)/2}\cdots y_\ell^{N_\ell+(p-1)/2}$ in $y_1^{(p-1)/2}\cdots y_\ell^{(p-1)/2}h_{U_A(\Z_p^\ell,d)}(y_1,\ldots, y_\ell)$, so as before Propositions~\ref{prop:B-U}(b) and ~\ref{prop:VK-F product} complete the proof.
\end{proof}

\section{$4$-Tverberg Families of Type $(2,2)$ in dimensions a power of two} 
\label{sec:join}

We conclude with our proof of the existence of balanced $(4,d;2,2)$-Tverberg admissible pairs when $d$ is a power of two, as well as the van Kampen--Flores variant of that result. Our proofs here reduce to a known $D_8$-equivariance statement~\cite{FS24} which is given for products of \emph{deleted joins} rather than deleted products. We shall therefore convert the configuration-space/test-map scheme of Section~\ref{sec:CS/TM} to match this scenario.

\subsection{Configuration Space} Let $n\geq 1$ be an integer. Viewing $\Delta_n$ as an abstract simplicial complex, the \emph{deleted join} $(\Delta_n)^{\ast 2}_\Delta=\{\sigma\ast\tau\mid \sigma,\tau\in \Delta_n\,\, \text{and}\,\, \sigma\cap \tau=\emptyset\}$ is the abstract simplicial complex consisting of all joins of disjoint faces of $\Delta_n$, including the possibility of empty faces. Viewed as a topological space, each point $\mathbf{x}=\lambda_1x_1\oplus \lambda_2x_2$ of $(\Delta_n)^{\ast 2}_\Delta$ can be represented as a formal convex sum, where $x_1,x_2\in \Delta_n$ lie in disjoint faces and $\lambda_1,\lambda_2\geq 0$ satisfy $\lambda_1+\lambda_2=1$. One then has a free $\Z_2$-action on $(\Delta_n)^{\ast 2}_\Delta$ generated by swapping $\lambda_1x_1$ and $\lambda_2x_2$. The deleted product $(\Delta_n)^{\times 2}_\Delta$ may be identified with the invariant subspace of the deleted join consisting of all elements with $\lambda_1=\lambda_2=1/2$. While we shall not need it in what follows, we note that there is a natural equivariant identification of $(\Delta_n)^{\ast 2}_\Delta$ with the ordinary join $\Z_2^{\ast (n+1)}$ of the group $\Z_2$ under the diagonal $\Z_2$-action (see, e.g.,~\cite{Ma08}). As the latter is the boundary of a $(n+1)$-dimensional cross-polytope, $(\Delta_n)^{\ast 2}_\Delta$ may therefore be identified with the $n$-dimensional sphere $S^n$ equipped with the standard antipodal action. 

Considering the product $(\Delta_n)^{\ast 2}_\Delta\times (\Delta_n)^{\ast 2}_\Delta$ of deleted joins, as before one has a free $\Z_2^2$-action by letting each $\Z_2$-factor of $\Z_2^2$ act independently on each corresponding join factor. This action can be extended to a non-free action of the Dihedral group $D_8=(\Z_2\times \Z_2)\rtimes \mathfrak{S}_2$ given by permuting the two deleted join copies. Explicitly, if $\mathbf{x}=(\mathbf{x}_1,\mathbf{x}_2)\in (\Delta_n)^{\ast 2}_\Delta \times (\Delta_n)^{\ast 2}_\Delta$, then one sets $(g,\varphi)\cdot \mathbf{x}=(g_1\cdot \mathbf{x}_{\varphi(1)}, g_2\cdot \mathbf{x}_{\varphi(2)})$ for each $g=(g_1,g_2)\in \Z_2^2$ and $\varphi\in \mathfrak{S}_2$. The restriction of this action to the product of deleted products gives a $D_8$-action on $(\Delta_n)^{\times 2}_\Delta\times (\Delta_n)^{\times 2}_\Delta$, which is again non-free. 

\subsection{Test Spaces} We now consider the $\Z_2^2$-representations of Section~\ref{sec:CS/TM}, which may be viewed as $D_8$-representations as follows. First, the permutative $\mathfrak{S}_2$-action on $\Z_2\oplus \Z_2$ extends to a $D_8$-action given by $(h,\varphi)\cdot g=h+\varphi\cdot g$ for each $h=(h_1,h_2), g=(g_1,g_2)\in \Z_2^2$ and $\varphi\in \mathfrak{S}_2$. Letting $(h,\varphi)\cdot \sum_{g\in \Z_2^2} x_g\, g= \sum_{g\in \Z_2^2}x_{h+\varphi\cdot g}\, g$ for each $\sum_{g\in \Z_2^2} x_g\,g\in U(\Z_2^2,d)$ then gives a $D_8$-action on $U(\Z_2^2,d)$. Secondly, for each $m\geq 1$ let $V_m:=[U(\Z_2)\oplus U(\Z_2)]^{\oplus m}$. As before, there is a $\Z_2^2$-representation on $U(\Z_2)\oplus U(\Z_2)$ by letting each $\Z_2$-factor act independently on each $U(\Z_2)$-factor, and this extends to a $D_8$-action by including transposition of the two coordinates. One then has a diagonal $D_8$-action on $V_m$.   

\subsection{Test Map} The following Borsuk-Ulam type statement for $D_8$-actions was given over the course of Sections 5 and 6 of~\cite{FS24}.

\begin{proposition}
\label{prop:D_8}
Let $s\geq 1$ and $m\geq 0$ be integers, let $d=2^s+t$ where $t\in \{0,1\}$, and let $n=m+\lceil\frac{3d}{2}\rceil+1$. Then any $D_8$-equivariant map $F\colon (\Delta_n)^{\ast 2}_\Delta\times (\Delta_n)^{\ast 2}_\Delta \rightarrow U(\Z_2^2,d)\oplus V_{m+1}$ has a zero.
\end{proposition}

\begin{remark}
\label{rem:m=0}
Owing to its use in a Tverberg-type theorem for maps from simplices to $\R^m$, Proposition~\ref{prop:D_8} is stated in ~\cite{FS24} only when $m\geq 1$. One may easily confirm that the proofs given there carry through without change to the $m=0$ setting, however. \end{remark}

We now prove Theorem~\ref{thm: arbitrary/balanced}(b), together with its van Kampen--Flores type extension. 

\begin{proof} Let $d\geq 2$ be a power of two and let $n=\frac{3d}{2}+1$. Given a continuous map $f\colon \Delta_n\times \Delta_n\rightarrow \R^d$, we have previously constructed a $\Z_2^2$-equivariant map $F\colon (\Delta_n)^{\times 2}_\Delta \times (\Delta_n)^{\times 2}_\Delta \rightarrow U(\Z_2^2,d)$, a zero of which guarantees that the map $f$ admits a $4$-Tverberg family of type $(2,2)$. It is now easily confirmed that this map is in fact $D_8$-equivariant with respect to the actions considered above. In order to use Proposition~\ref{prop:D_8}, we extend the map $F$ to a $D_8$-equivariant map $\widetilde{F}\colon (\Delta_n)^{\ast 2}_\Delta \times (\Delta_n)^{\ast 2}_\Delta\rightarrow U(\Z_2^2,d)\oplus V_1$ whose zeros coincide with those of $F$. This is done by setting

\begin{align*} \widetilde{F}(\lambda_1x_1\oplus\lambda_2x_2,\mu_1y_1\oplus\mu_2y_2)=&\lambda_1\lambda_2\mu_1\mu_2 F((x_1,x_2),(y_1,y_2))
												 \oplus ((\lambda_1-1/2)\cdot 0+(\lambda_2-1/2)\cdot 1)\\ &\oplus ((\mu_1-1/2)\cdot 0 +(\mu_2-1/2)\cdot 1)\end{align*} for each $\lambda_1x_1\oplus \lambda_2x_2,\mu_1y_1\oplus \mu_2y_2\in (\Delta_n)^{\ast 2}_\Delta$. This map has a zero by Proposition~\ref{prop:D_8}, and therefore $F$ does as well. This concludes the proof of Theorem~\ref{thm: arbitrary/balanced}(b).

The proof of the van Kampen--Flores extension of Theorem~\ref{thm: arbitrary/balanced}(b) is analogous. Given a map $f\colon \Delta_{n+1}\times \Delta_{n+1}\rightarrow \R^d$, the map $F\oplus D\colon (\Delta_{n+1})^{\times 2}_\Delta\times (\Delta_{n+1})^{\times 2}_\Delta\rightarrow U(\Z_2^2,d)\oplus V_1$ constructed in Section~\ref{sec:CS/TM} is again $D_8$-equivariant. One may extend this map to a $D_8$-equivariant map $\widetilde{F\oplus D}\colon (\Delta_{n+1})^{\ast 2}_\Delta \times (\Delta_{n+1})^{\ast 2}_\Delta\rightarrow U(\Z_2^2,d)\oplus V_2$ in the same way as before. This map has a zero by Proposition~\ref{prop:D_8}, which completes the proof because the zero sets of $\widetilde{F\oplus D}$ and $F\oplus D$ coincide.\end{proof}

\begin{remark} 
\label{rem:d=2^s+1} As Proposition~\ref{prop:D_8} holds for $d$ of the form $d=2^s+1$ for some integer $s\geq 2$, the same proof as above shows that when $n=\lceil\frac{3d}{2}\rceil+1$ any continuous map $f\colon \Delta_n\times\Delta_n\rightarrow \R^d$ admits a $4$-Tverberg family of type $(2,2)$. Likewise, any continuous map $f\colon \Delta_{n+1}\times \Delta_{n+1}\rightarrow \R^d$ admits a $4$-Tverberg family $\{\sigma_{i_1}^1\times \sigma_{i_2}^2\}_{i_1,i_2\in[2]}$ of type $(2,2)$ with $\dim(\sigma^j_i)\leq \lfloor\frac{n}{2}\rfloor$ for all $i,j\in [2]$. 
\end{remark}

\end{document}